\documentclass[11pt,letterpaper]{article}
\usepackage{macros}
\usepackage{xcolor}
\usepackage{tikz}
\usetikzlibrary{arrows.meta, automata, positioning}

\usepackage[notes=true,later=false,camera=true]{dtrt}

\newcommand{\etal}{et.~al.\xspace}
\newcommand{\cR}{\calR} 

\allowdisplaybreaks

\begin{document}
\title{Ellipsoid Fitting Up to a Constant}

 \author{Jun-Ting Hsieh\thanks{Carnegie Mellon University.  \texttt{juntingh@cs.cmu.edu}.  Supported by NSF CAREER Award \#2047933.} \and
 Pravesh K. Kothari\thanks{Carnegie Mellon University, \texttt{praveshk@cs.cmu.edu}. Supported by  NSF CAREER Award \#2047933, Alfred P. Sloan Fellowship and a Google Research Scholar Award.}
 \and Aaron Potechin\thanks{The University of Chicago, \texttt{potechin@uchicago.edu}. Supported in part by NSF grant CCF-2008920. }
 \and
 Jeff Xu\thanks{Carnegie Mellon University, \texttt{jeffxusichao@cmu.edu}. Supported in part by NSF CAREER Award \#2047933, and a Cylab Presidential Fellowship.}}

\maketitle

\begin{abstract}
In \cite{Sau11,SPW13}, Saunderson, Parrilo and Willsky asked the following elegant geometric question: what is the largest $m= m(d)$ such that there is an ellipsoid in $\R^d$ that passes through $v_1, v_2, \ldots, v_m$ with high probability when the $v_i$s are chosen independently from the standard Gaussian distribution $N(0,I_{d})$. The existence of such an ellipsoid is equivalent to the existence of a positive semidefinite matrix $X$ such that $v_i^{\top}X v_i =1$ for every $1 \leq i \leq m$ --- a natural example of a \emph{random} semidefinite program. SPW conjectured that $m= (1-o(1)) d^2/4$ with high probability. Very recently, Potechin, Turner, Venkat and Wein~\cite{PTVW22} and Kane and Diakonikolas~\cite{KD22} proved that $m \geq d^2/\log^{O(1)}(d)$ via certain explicit constructions.

In this work, we give a substantially tighter analysis of their construction to prove that $m \geq d^2/C$ for an absolute constant $C>0$. This resolves one direction of the SPW conjecture up to a constant. Our analysis proceeds via the method of \emph{Graphical Matrix Decomposition} that has  recently been used to analyze correlated random matrices arising in various areas~\cite{BHK+19}. Our key new technical tool is a refined method to prove singular value upper bounds on certain correlated random matrices that are tight up to absolute dimension-independent constants. In contrast, all previous methods that analyze such matrices lose logarithmic factors in the dimension.

\end{abstract}

\thispagestyle{empty}
\setcounter{page}{0}
\thispagestyle{empty}
\newpage
\setcounter{page}{0}
\tableofcontents
\thispagestyle{empty}
\newpage

\section{Introduction}

What's the largest $m$ so that for $m$ points $v_1, v_2, \ldots, v_m \in \R^d$ sampled independently from the $d$-dimensional standard Gaussian distribution $\calN(0,I_d)$, there exists an ellipsoid that passes through each of the $v_i$s with high probability? This latter condition is equivalent to asking for a positive semidefinite matrix $\Lambda$ such that $v_i^{\top} \Lambda v_i = 1$ for $1 \leq i \leq m$ and thus, equivalently, the question asks for the largest $m$ such that the basic stochastic semidefinite program above remains feasible with high probability.  

It is not hard to prove that for any $m\leq d+1$, an ellipsoid as above exists with high probability over the $v_i$s~\cite{SCPW12}. On the other hand, since the dimension of the smallest linear subspace that contains the positive semidefinite cone of $d \times d$ matrices is ${{d+1} \choose 2}\sim d^2/2$, it is easy to prove that for $m \gg d^2/2$, there cannot be an ellipsoid passing through $v_i$s with high probability. 

In 2013, Saunderson, Parrilo and Willsky~\cite{SPW13} studied this basic geometric question and conjectured that there is a sharp phase transition for the problem (from feasibility/existence of an ellipsoid to non-existence of an ellipsoid) as $m$ crosses $d^2/4$. 
\begin{conjecture}[SCPW conjecture]
    Let $\eps>0$ be a constant and $v_1,\dots, v_m \sim \calN(0,I_d)$ be i.i.d.\ standard Gaussian vectors in $\R^d$. Then,
    \begin{enumerate}
        \item If $m \leq (1-\eps) \frac{d^2}{4}$, then $v_1,\dots, v_m$ have the ellipsoid fitting property with probability $1-o_d(1)$.
        \item If $m \geq (1+\eps) \frac{d^2}{4}$, then $v_1,\dots, v_m$ have the ellipsoid fitting property with probability $o_d(1)$.
    \end{enumerate}
\end{conjecture}

This bound is $1/2$ of the dimension of the smallest linear subspace containing the positive semidefinite cone of $d \times d$ matrices. Said differently, the SCPW conjecture (developed in a sequence of works~\cite{Sau11,SCPW12,SPW13}) posits that the positive semidefiniteness constraint ``forces" a drop of a factor $2$ in the threshold $m$ for infeasibility. The SCPW conjecture was motivated by results of numerical experiments (see also the experiments presented in \cite{PTVW22}). 

Early on, \cite{Sau11,SPW13} established the existence of a feasible ellipsoid for any $m \leq O(d^{6/5-\eps})$ whp. Recently, there has been a new wave of progress on this bound.  A recent result ~\cite{GJJ20} on establishing Sum-of-Squares lower bounds for the Sherringtin-Kirkpatrick model, as a corollary yields an estimate of $m \leq O(d^{3/2-\eps})$. In fact, though not explicitly stated, their work already contains ideas that imply a significantly stronger bound of $m \leq d^{2}/\polylog(d)$. Very recently, two independent works \cite{PTVW22} and \cite{KD22} analyzed two slightly different explicit constructions for $\Lambda$ to recover a similar bound of $m \leq d^{2}/\polylog(d)$. In their works~\cite{PTVW22,KD22}, the authors ask the question of analyzing their construction (or a different one) to obtain an improved and almost optimal estimate of $m = d^2/C$ for some absolute constant $C>0$.

The main result of this work achieves this goal. Specifically, we prove:



\begin{theorem}[Main result] \label{thm:main-thm}
    For $m \leq c d^2$ for some universal constant $c>0$ and $v_1,\dots,v_m \sim \calN(0, I_d)$ drawn independently, with probability at least $1-o_d(1)$, there exists an ellipsoid passing through each $v_i$.
\end{theorem}

\begin{remark}
    We note that the failure probability is in fact $2^{-d^{\eps}}$ for some small constant $\eps$ due to the nature of our proof for matrix norm bounds.
\end{remark}

We establish \Cref{thm:main-thm} by analyzing the construction of Kane and Diakonikolas~\cite{KD22} (which is a variant of the construction proposed in~\cite{PTVW22}). Our argument can be used to recover a bound of $c \sim 1/10^8$. We have not tried to optimize this bound. Numerical experiments suggest that the \cite{KD22} construction we analyze cannot approach $c=1/4$, so establishing the sharp constant in the SCPW conjecture will likely need new ideas. \Cref{tab:prior-work} shows a summary of our result compared to prior work.

Our key idea departs from the analysis technique of~\cite{KD22} and instead relies on the method of \emph{graphical matrix decomposition}. This method decomposes a random matrix with correlated entries into a sum of structured random matrices called \emph{graph matrices}. Graph matrices can be thought of as an analog of the \emph{Fourier basis} in the analysis of functions over product spaces. This method was first employed in the works establishing tight sum-of-squares lower bound on the planted clique problem~\cite{BHK+19,HKP19,AMP16,JPR+22} and has since then been employed in several follow-up works on proving sum-of-squares lower bounds and more recently in analyzing well-conditionedness of linear algebraic algorithms for generalizations of tensor decomposition~\cite{BHKX22}).

The key technical work in the analysis then becomes understanding the smallest and the largest singular values of graph matrices. All prior works rely on arguments that establish bounds on the largest singular values that are accurate up to polylogarithmic factors in the underlying dimension of the matrices. The work of \cite{BHKX22} recently showed how to use such bounds to also obtain estimates of the smallest singular values of graph matrices (which, otherwise are significantly more challenging to prove). Our analysis builds on their conceptual framework but with significant technical upgrades. This is because the quantitative bounds proved in \cite{BHKX22} do not allow us to directly obtain an improvement on the previous estimates~\cite{KD22}. 

Our main technical contribution is a new method to establish bounds on the largest singular values of graph matrices that are tight up to dimension-independent absolute constants. This allows us to obtain substantially improved estimates for the SCPW conjecture. Given the host of previous applications of such bounds, we expect that our results will have many more applications down the line. 

\paragraph{Concurrent Work} We note that a concurrent work of Bandeira \etal \cite{Bandeira2023FittingAE} also obtains a sharper analysis of \cite{KD22} to establish a similar result as this work.  They analyze the same construction of identity perturbation as us. In their work, \cite{Bandeira2023FittingAE} ask the question of obtaining estimates that hold with inverse exponential failure probability (as opposed to inverse polynomial failure probability that their work establishes). They also outline a proof strategy that could potentially achieve this goal. We note that our proof does indeed recover an inverse exponential failure probability naturally.


\begin{table}[ht]
\begin{center}
\begin{tabular}{lc}
\hline
 {\textbf{Construction}} &
 { \textbf{Bound on \emph{m}}} 
 \\ \hhline{==}
 { Conjectured}  & {$d^{2}/4$}      \\
 { \cite{Sau11,SPW13}}    & {$O(d^{6/5-\eps})$}   \\
 { \cite{GJJ20}}    & {$O(d^{3/2-\eps})$ ${}^{\ast}$}  \\
 { \cite{PTVW22}}  & {$O(d^2/\polylog(d))$}  \\
 { \cite{KD22}} & {$O(d^2/\log^4(d))$}   \\  \hline
 { \textbf{this paper}} & {$O(d^2)$}  \\ \hline
\end{tabular}
\end{center}
\centering \footnotesize{${}^{\ast}$The bound $O(d^2/\polylog(d))$ is implicit in their work.}
\caption{Comparison of our result with prior work.}
\label{tab:prior-work}
\end{table}

\subsection{Technical overview}

Following the convention of \cite{KD22}, for the rest of the paper we will assume that $v_1,\dots,v_m \sim \calN(0, \frac{1}{d} I_d)$ such that each vector has expected norm $1$.
Note that this does not change the problem as we can simply scale $\Lambda$.

Our construction of $\Lambda$ is the ``identity perturbation construction'', which is the same one analyzed in \cite{KD22} and was proposed in \cite{PTVW22}.
As an intuition, observe that $\Lambda = I_d$ almost works: $v_i^T I_d v_i = \|v_i\|_2^2 \approx 1$.
Thus, the idea is to define $\Lambda$ as a perturbation of $I_d$: $\Lambda = I_d - \sum_{i=1}^m w_i v_i v_i^T$, where $w = (w_1,\dots,w_m) \in \R^m$.
To determine $w$, observe that the constraints $v_i^T \Lambda v_i = 1$ give $m$ linear constraints on $w$, and this can be written as a linear system represented by a matrix $M \in \R^{m\times m}$ with entries $M[i,j] = \iprod{v_i, v_j}^2$.
Thus, given that $M$ is full rank, $w$ is uniquely determined by $w = M^{-1} \eta$ for some vector $\eta$ (see \Cref{eq:eta-vector}).
This construction satisfies $v_i^T \Lambda v_i = 1$ automatically, so the next thing is to prove that $\Lambda \succeq 0$.
Therefore, we have two high-level goals:
\begin{enumerate}
    \item Prove that $M$ is full rank and analyze $M^{-1}$.
    \item Prove that $\calR \coloneqq \sum_{i=1}^n w_i v_i v_i^T$ has spectral norm bounded by $1$.
\end{enumerate}
These two immediately imply that $\Lambda$ is a valid construction.

To achieve the first goal, we \emph{decompose} $M$ into several components.
Roughly, we write $M = A + B$ where $A$ is a perturbed identity matrix $A = I_m - T$ and $B$ is a rank-2 matrix (see \Cref{sec:M-decomposition}).
We first show that $\|T\|_{\op} \leq O(\frac{\sqrt{m}}{d}) < 0.5$ with $m \leq c d^2$ (\Cref{lem:T-norm}), hence $A$ is well-conditioned.
Then, using the fact that $B$ has rank 2, we can apply the Woodbury matrix identity (\Cref{fact:invertibility} and \cref{fact:woodbury}) --- a statement on the inverse of low-rank corrections of matrices --- to conclude that $M$ is invertible and obtain an expression for $M^{-1}$ (\Cref{eq:M-inverse-expanded}).
This is carried out in \Cref{sec:M-inverse}.

Next, for the second goal, we need to further expand $A^{-1}$.
Since $\|T\|_{\op} < 1$, we can apply the Neumann series and write $A^{-1} = (I_m - T)^{-1} = \sum_{k=0}^\infty T^k$.
For the analysis, we select certain thresholds to truncate this series such that the truncation error is small.
Then, we write $M^{-1}$ in terms of the truncated series plus a small error, which will be useful later for the analysis of $R$.
This is carried out in \Cref{sec:A-inverse}.

Finally, given the expression of $M^{-1}$, 
we are able to express $\calR$ using the terms that show up for $M^{-1}$, and the bulk of our work culminates in bounding the spectral norm of $\calR$ in \cref{sec:analysis-of-R}.
Bounding $\|R\|_{\op} \leq 1$ implies that $\Lambda \succeq 0$, completing the proof.

\paragraph{Requiring tight norm bounds}
Our main technical lemmas are the spectral norm bounds of $T$ (\Cref{lem:T-norm}) and the matrices in the decomposition of $R$ at  \cref{sec:analysis-of-R}.
Clearly, we need our norm bound $\|T\|_{\op} \leq O(\frac{\sqrt{m}}{d})$ to be tight without polylog factors so that $m \leq O(d^2)$ suffices, and similarly for matrices from $R$.

The standard starting point is the \emph{trace moment method}: for any symmetric matrix $M \in \R^{n\times n}$ and $q\in \N$ (usually taking $q = \polylog(n)$ suffices),
\begin{align*}
    \|M\|_{\op}^{2q} \leq \tr( M^{2q}) = \sum_{i_1,i_2,\dots,i_{2q}\in [n]} M[i_1,i_2] M[i_2,i_3] \cdots M[i_{2q}, i_1] \mper
\end{align*}
We view the summand as a closed walk $i_1 \to i_2 \to \cdots \to i_{2q} \to i_1$ on $n$ vertices.
For a random matrix, we study the expected trace $\E \tr(M^{2q})$.
In the simple case when $M$ is a Gaussian matrix (GOE), we see that after taking the expectation, the non-vanishing terms are closed walks where each edge $(u,v)$ is traversed even number of times.
This is in fact true for any $M$ as long as the odd moments are zero.
Thus, a precise upper bound on $\E \tr(M^{2q})$ can be obtained by carefully counting such closed walks (see \cite{Tao12}).

Our matrices are more complicated; each entry is a mean-zero \emph{polynomial} of Gaussian random variables.
To carry out the trace method, we represent the matrices as graphs, hence the term \emph{graph matrices}.
The framework of graph matrices was first introduced by \cite{BHK+19}, and
over the years, off-the-shelf norm bounds (e.g.~\cite{AMP16}) for graph matrices have been developed and successfully used in several works~\cite{MRX20,GJJ20,HK22,JPR+22,BHKX22}.
However, the currently known norm bounds are only tight up to polylog factors, hence not sufficient for us.
Therefore, the bulk of our paper is to prove norm bounds for these matrices that are tight up to constant factors. In fact, our bounds are even tight in the constant when the matrices are explicitly written down following the graph matrix language. That said, we do not pursue the tight constant-factor dependence in this work: we believe that an analysis of our candidate matrix following the current road-map but with norm bounds tight-to-constant would still fall short of reaching the conjectured threshold of $\frac{d^2}{4}$.

In the context of a fine-grained understanding for graph matrices,  Potechin and Cai \cite{Cai2020TheSO, Cai2022OnMD} determined the limiting distribution of the spectrum of the singular values of Z-shaped and multi-Z-shaped graph matrices. However, their results are only for these specific graph matrices, and their analysis does not technically give norm bounds as they do not rule out having a negligible proportion of larger singular values.

\paragraph{Key idea towards tight norm bounds}
Here, we briefly discuss the high-level ideas for proving tight norm bounds.
To illustrate our techniques, in \Cref{sec:overview-norm-bounds} we will give a full proof for a matrix that arises in our analysis as an example, and also discuss key ideas that allow us to analyze more complicated matrices.

The key to counting walks is to specify an \emph{encoding}, which we view as \emph{information} required for a \emph{walker} to complete a walk.
If we can show that such an encoding uniquely identifies a walk, then we can simply bound the walks by bounding the number of possible encodings.
Thus, all we need to do is to come up with an (efficient) encoding scheme and prove that the walker is able to complete a walk.
Using standard encoding schemes, we quickly realize that the walker may be \emph{confused} during the walk, i.e., the walker does not have enough information to perform the next step.
Thus, we need to \emph{pay} for additional information in the encoding to resolve confusions.
So far, this is the same high-level strategy that was used in prior work~\cite{Tao12,AMP16,JPR+22}, and this extra pay is often the source of extra log factors in the norm bounds.

Our key innovation is to pay for the extra information during steps that require much less information than normal.
Roughly speaking, we label each step of the walk as either (1) visiting a new vertex, (2) visiting an old vertex via a new edge, (3) using an old edge but not the last time, (4) using an old edge the last time (see \Cref{def:step-labeling}).
The high level idea is that the dominating walks in the trace are the ones that use only the 1st and 4th types, while the 2nd and 3rd types require less information (which we call \emph{gaps}).
The main observation is that the walker will be confused only when there are steps of the 2nd and 3rd type involved, but we can pay extra information during these steps to resolve potential (future) confusions.
This is illustrated in \Cref{sec:resolving-confusion}.

\subsection{Comparison to prior works}

Our candidate matrix construction of $\Lambda$ is essentially same as \cite{KD22}, while we adopt different techniques to bound the spectral norm of the non-Identity component.
In particular,  they use an elegant cover (or $\eps$-net) argument which is significantly different than ours.
That said, though a major obstacle being the norm bound for the invertibility, their argument suffers an additional polylog gap from the epsilon-net argument, and this is partially why we adopt the proof strategy via graph matrix decomposition that is seemingly more complicated.

Closer to our analysis is the work of \cite{PTVW22}. They study a construction of ``least-square minimization'' proposed by \cite{Sau11}, which is equivalent to projecting out the identity mass onto the subspace of matrices satisfying the constraints. In particular, their matrix analysis proceeding via Woodbury expansion and Neumann series using graph matrices serves as a road-map for our current work. In this work, we develop a more refined understanding of the structured random matrices that we believe would be useful in further and more fine-grained investigations of problems in average-case complexity.


 In the context of Planted Affine Plane problem, \cite{GJJ20} reaches the threshold of $\widetilde{O}(d^2 )$ implicitly. They adopt the framework of pseudo-calibration \cite{BHK+19} to obtain a candidate matrix, and follow a similar recipe as ours via graph matrix decompositions and spectral analysis. 
That said, it is an interesting question whether solutions coming from a pseudo-calibration type of construction might give us some extra mileage in ultimately closing the constant gap.

\begin{remark}[A ''quieter'' planted distribution?]
  A natural idea is to analyze the planted distribution pioneered in \cite{MRX20, GJJ20}:  unfortunately, it can be easily verified that the low-degree polynomial hardness for the particular planted distribution actually falls apart even if we assume an arbitrary constant gap.
\end{remark}



\section{Proof of main result}

Given $v_1, v_2,\dots, v_m$ that are i.i.d.\ samples from $\calN(0,\frac{1}{d}I_d)$, recall that we must construct a matrix $\Lambda$ such that (1) $v_i^T \Lambda v_i = 1$ for any $i\in[m]$, and (2) $\Lambda \succeq 0$.

In this section, we describe our candidate matrix (\Cref{def:candidate-matrix}).
To prove that it satisfies the two conditions above, we need to analyze certain random matrices (and their inverses) that arise in the construction, which involves decomposing the matrices into simpler components.
We will state our key spectral norm bounds (\Cref{lem:T-norm} and \Cref{lem:R-norm-bound}) whose proofs are deferred to later sections,
and complete the proof of \Cref{thm:main-thm} in \Cref{sec:finish-main-thm}.

\subsection{Candidate construction}

The following is our candidate matrix $\Lambda$, which is the same one as \cite{KD22}.

\begin{definition}[Candidate matrix]  \label{def:candidate-matrix}
    Given $v_1,\dots,v_m \sim \calN(0, \frac{1}{d} I_d)$, we define the matrix $\Lambda \in \R^{d \times d}$ to be
    \[ 
        \Lambda \coloneqq I_d - \sum_{i=1}^m w_i v_i v_i^T
        \numberthis \label{eq:lambda-matrix}
    \]
    where we take $w = (w_1, w_2,\dots, w_m)$ to be the solution to the following linear system,
    \[ 
        M w = \eta 
    \]
    for $\eta\in \R^{m} $ given by
    \[ 
        \eta_i \coloneqq \|v_i\|_2^2-1, \quad \forall i\in [m] \mcom
        \numberthis \label{eq:eta-vector}
    \]
    and $M\in \R^{m\times m}$ with entries given by
    \[ 
        M[i,j] \coloneqq \langle v_i, v_j \rangle^2, \quad \forall i,j\in [m] \mper
        \numberthis \label{eq:M-matrix}
    \] 
\end{definition}

We first make the following simple observation.
\begin{observation}
    For any $i\in [m]$, the constraint $v_i^T \Lambda v_i = 1$ is satisfied.
\end{observation}
\begin{proof}
    For any $i\in [m]$,
    \begin{align*}
        v_i^T \Lambda v_i &=  v_i^T I_d v_i - \sum_{j\in [m]} w_j \langle v_i, v_j\rangle^2  
        = \|v_i\|_2^2 - \langle M[i], w\rangle 
        = \|v_i\|_2^2 - \eta_i =1 \mper
    \end{align*}
    Here $M[i]$ is the $i$-th row of $M$, and the equality above follows from $M w = \eta$ and $\eta_i = \|v_i\|_2^2 - 1$ from \Cref{eq:eta-vector}.
\end{proof}

\paragraph{Structure of subsequent sections}
For $\Lambda$ to be well-defined, we require that $M$ is full rank (hence invertible).
Note that it is easy to see that $M$ is positive semidefinite, since $M$ is a Gram matrix with $M[i,j] = \iprod{v_i^{\otimes 2}, v_j^{\otimes 2}}$.
To analyze $M$, we will show a decomposition of $M$ in \Cref{sec:M-decomposition} that allows us to more easily analyze its inverse.
In \Cref{sec:M-inverse}, we will prove that $M$ is in fact positive definite (\Cref{lem:M-invertible}).

Next, to prove that $\Lambda \succeq 0$,
we will write $\Lambda = I_d - \cR$ where
\begin{equation*}
    \cR \coloneqq \sum_{i=1}^m w_i v_iv_i^T = \sum_{i=1}^m \left( M^{-1} \eta \right)[i] \cdot v_iv_i^T \mcom
    \numberthis \label{eq:R-matrix}
\end{equation*}
and prove that $\|\cR\|_{\op}$ is bounded by 1.
This is done in \Cref{sec:R-decomposition}.
Finally, combining the analyses, we finish the proof of \Cref{thm:main-thm} in \Cref{sec:finish-main-thm}.


\subsection{Decomposition of \texorpdfstring{$M$}{M}}
\label{sec:M-decomposition}

The proof of \Cref{thm:main-thm} requires careful analysis of the matrix $M$ from \Cref{eq:M-matrix} and its inverse.
To this end, we first decompose $M$ as $M = A + B$ such that intuitively, $A$ is perturbation of a (scaled) identity matrix and $B$ has rank 2.
We will later see how this decomposition allows us to analyze $M^{-1}$ more conveniently.


\begin{proposition}[Decomposition of $M$]  \label{prop:M-decomposition}
    \[ 
    M = \underbrace{M_\al + M_\beta +M_D + \Paren{1+\frac{1}{d}} I_m}_{\coloneqq A}
    + \underbrace{\frac{1}{d}J_m + \frac{1}{d} \left( 1_m \cdot \eta^T + \eta\cdot 1_m^T\right) }_{\coloneqq B}
    \numberthis  \label{eq:M-decomposition}
    \]
    where $J_m$ is the all-ones matrix, $M_{\alpha}, M_{\beta} \in \R^{m\times m}$ are matrices with zeros on the diagonal and $M_D \in \R^{m\times m}$ is a diagonal matrix, defined as follows:
    \begin{itemize}
        \item $M_{\alpha}[i,j] \coloneqq \sum_{a \neq b \in [d]} v_i[a]\cdot  v_i[b] \cdot v_j[a]\cdot v_j[b]$ for $i \neq j \in [m]$,
        \item $M_{\beta} [i,j] \coloneqq \sum_{a \in [d]} \left( v_i[a]^2 -\frac{1}{d} \right)\left( v_j[a]^2  - \frac{1}{d} \right)$ for $i \neq j \in [m]$,
        \item $M_D[i,i] \coloneqq \|v_i\|_2^4 - \frac{2}{d} \|v_i\|_2^2 - 1$ for $i\in[m]$.
    \end{itemize}
\end{proposition}
\begin{proof}
    For any off-diagonal entry $i\neq j \in [m]$, on the right-hand side we have
    \begin{align*}
        M[i,j] &= \iprod{v_i, v_j}^2
        = \Paren{\sum_{a\in[d]} v_i[a] v_j[a] }^2 \\
        &= \sum_{a\neq b\in[d]} v_i[a]\cdot v_i[b]\cdot v_j[a]\cdot v_j[b]
          + \sum_{a\in[d]} v_i[a]^2 \cdot v_j[a]^2 \mper
    \end{align*}
    The first term is exactly $M_{\alpha}[i,j]$.
    For the second term,
    \begin{align*}
        \sum_{a\in[d]} v_i[a]^2 \cdot v_j[a]^2
        &= \sum_{a\in[d]} \Paren{v_i[a]^2 - \frac{1}{d}} \Paren{v_j[a]^2 - \frac{1}{d}} + \frac{1}{d} \Paren{\|v_i\|_2^2 + \|v_j\|_2^2} - \frac{1}{d} \\
        &= \underbrace{\sum_{a\in[d]} \Paren{v_i[a]^2 - \frac{1}{d}} \Paren{v_j[a]^2 - \frac{1}{d}} }_{M_{\beta}[i,j]}
          + \underbrace{ \frac{\|v_i\|_2^2 - 1}{d} }_{\frac{1}{d}\eta_i}
          + \underbrace{ \frac{\|v_j\|_2^2 - 1}{d} }_{\frac{1}{d}\eta_j} + \frac{1}{d} \mper
    \end{align*}
    Thus, $M[i,j] = M_{\alpha}[i,j] + M_{\beta}[i,j] + \frac{1}{d} + \frac{1}{d} \left( 1_m \cdot \eta^T + \eta\cdot 1_m^T\right)[i,j]$.

    For the diagonal entries, the right-hand side of the $(i,i)$ entry is
    \begin{align*}
        M_D[i,i] + \Paren{1+\frac{1}{d}} + \frac{1}{d} + \frac{2}{d} \eta_i
        &= \Paren{ \|v_i\|_2^4 - \frac{2}{d} \|v_i\|_2^2 - 1 } + 1 + \frac{2}{d} + \frac{2}{d} (\|v_i\|_2^2 -1) \\
        &= \|v_i\|_2^4 = M[i,i] \mper
    \end{align*}
    This completes the proof.
\end{proof}


\begin{remark}
    The intention behind this decomposition is so that for $v_i \sim \calN(0, \frac{1}{d}I_d)$, $M_{\alpha}, M_{\beta}, M_D$ are all mean $0$ (while not the same variance) since $\E \|v_i\|_2^2 = 1$ and $\E \|v_i\|_2^4 = 1 + \frac{2}{d}$.
    Therefore, we expect $\|M_{\alpha} + M_{\beta} + M_D\|_{\op}$ to be small, which implies that $A$ is positive definite and well-conditioned.
    Furthermore, observe that $B$ has rank 2:
    \begin{equation*}
        B = \frac{1}{d}J_m + \frac{1}{d} \left( 1_m \cdot \eta^T + \eta\cdot 1_m^T\right)
        = \frac{1}{d}\begin{bmatrix}
            1_m & \eta
        \end{bmatrix}
        \cdot 
        \begin{bmatrix}
            1 & 1 \\ 1 & 0
        \end{bmatrix}
        \cdot
        \begin{bmatrix}
            1_m \\ \eta
        \end{bmatrix}
        \mper \numberthis \label{eq:B-matrix}
    \end{equation*}
\end{remark}

\subsection{Inverse of \texorpdfstring{$M$}{M}}
\label{sec:M-inverse}

The decomposition of $M$ into $A$ and a rank-2 matrix $B$ (\Cref{eq:M-decomposition}) allows us to apply the Woodbury matrix identity about the inverse of low-rank corrections of invertible matrices.

\begin{fact}[Matrix Invertibility] \label{fact:invertibility}
    Suppose $A \in \R^{n_1 \times n_1}$ and $C \in \R^{n_2 \times n_2}$ are both invertible matrices, and $U\in \R^{n_1 \times n_2}$ and $V \in \R^{n_2 \times n_1}$ are arbitrary.
    Then, $A + U C V$ is invertible if and only if $C^{-1} + V A^{-1} U$ is invertible.
\end{fact}

\begin{fact}[Woodbury matrix identity~\cite{Woo59}]  \label{fact:woodbury}
    Suppose $A \in \R^{n_1 \times n_1}$ and $C \in \R^{n_2 \times n_2}$ are both invertible matrices, and $U\in \R^{n_1 \times n_2}$ and $V \in \R^{n_2 \times n_1}$ are arbitrary. Then
    \[ 
    (A+ UCV)^{-1} = A^{-1} - A^{-1}U\left(C^{-1}+ VA^{-1}U\right)^{-1}VA^{-1} \mper
    \]
\end{fact}

In light of \Cref{fact:woodbury}, we can write $B$ in \Cref{eq:B-matrix} as $B = UCU^T$ where $U = V^T = \frac{1}{\sqrt{d}} \begin{bmatrix} 1_m & \eta \end{bmatrix} \in \R^{m \times 2}$ and
$C = \begin{bmatrix} 1 & 1 \\ 1 & 0 \end{bmatrix}$,
and $M = A + UCU^T$.
Note that $C^{-1} = \begin{bmatrix} 0 & 1 \\ 1 & -1 \end{bmatrix}$, and we have
\begin{align*}
    C^{-1} + U^T A^{-1}U  = \begin{bmatrix}
	\frac{1_m^T A^{-1} 1_m}{d} & 1 + \frac{\eta^T A^{-1} 1_m}{d}\\
	1+ \frac{\eta^T A^{-1} 1_m}{d} & -1 + \frac{\eta^T A^{-1} \eta }{d} 
    \end{bmatrix}
    \eqqcolon
    \begin{bmatrix}
    	r & s \\
    	s & u
    \end{bmatrix}
    \mper 
    \numberthis \label{eq:russ}
\end{align*}

We first need to show that $A$ is invertible.
Recall from \Cref{eq:M-decomposition} that $A = (1+\frac{1}{d})I_m + M_{\alpha} + M_{\beta} + M_D$.
We will prove the following lemma, whose proof is deferred to \Cref{sec:wrapping-up}.

\begin{lemma}[$M_{\alpha},M_{\beta},M_D$ are bounded]
\label{lem:T-norm}
    Suppose $m \leq cd^2$ for a small enough constant $c$.
    With probability $1-o_d(1)$, we have
    \begin{enumerate}
        \item $\|M_{\alpha}\|_{\op} \leq 0.1$,
        \item $\|M_{\beta}\|_{\op} \leq 0.1$,
        \item $\|M_{D}\|_{\op} \leq O\bigparen{\sqrt{\frac{\log d}{d}}}$.
    \end{enumerate}
\end{lemma}

As an immediate consequence, we get the following:
\begin{lemma}[$A$ is well-conditioned] \label{lem:A-eigenvalues}
    With probability $1-o_d(1)$, the matrix $A$ from \Cref{eq:M-decomposition} is positive definite (hence full rank), and
    \begin{equation*}
        0.5 I_m \preceq A \preceq 1.5 I_m \mper
    \end{equation*}
\end{lemma}
\begin{proof}
    Since $A = (1+\frac{1}{d})I_m + M_{\alpha} + M_{\beta} + M_D$, by \Cref{lem:T-norm} the eigenvalues of $A$ must lie within $1 \pm 0.2 \pm \widetilde{O}(1/\sqrt{d}) \in (0.5, 1.5)$ (we assume $d$ is large).
\end{proof}


Next, from \Cref{fact:invertibility}, we can prove that $M$ is invertible (\Cref{lem:M-invertible}) by showing that the $2\times 2$ matrix $C^{-1} + U^T A^{-1}U$ is invertible, which is in fact equivalent to $ru - s^2 \neq 0$.
We first need the following bound on the norm of $\eta$, whose proof is deferred to \Cref{sec:missing-proofs}.

\begin{claim}\label{Claim:eta-bound}
 With probability at least $1-o_d(1)$, \[ 
 \|\eta \|_2^2 \leq (1+o_d(1)) \frac{2m}{d} \mper
 \]
 \end{claim}

 \begin{lemma}[Bounds on $r,s,u$; $M$ is invertible]  \label{lem:M-invertible}
     Suppose $m \leq cd^2$ for a small enough constant $c$.
     Let $r,s,u \in \R$ be defined as in \Cref{eq:russ}.
     With probability at least $1-o_d(1)$, we have
     \begin{enumerate}
         \item $r \in \frac{m}{d} \cdot [2/3, 2]$,
         \item $|s| \leq 1+o_d(1)$,
         \item $u \in [-1, -1/2]$.
     \end{enumerate}
     Thus, we have
     \begin{align*}
         s^2 - ru \geq \Omega\Paren{\frac{m}{d}} \mper
     \end{align*}
     As a consequence, $M$ is invertible.
 \end{lemma}
\begin{proof}
    By \Cref{lem:A-eigenvalues}, we know that $\frac{2}{3} I_m \preceq A^{-1} \preceq 2I_m$.
    Thus, $r = \frac{1}{d} 1_m^T A^{-1} 1_m \in \frac{1}{d} \|1_m\|_2^2 \cdot [2/3,2]$, hence $r \in \frac{m}{d} \cdot [2/3, 2]$.

%
    For $u$, we have
    \begin{align*}
        \frac{1}{d} \Abs{\eta^T A^{-1} \eta} \leq  \frac{1}{d} \| A^{-1}\|_{\op} \cdot \|\eta\|_2^2 < (1+o_d(1))\cdot \frac{4m}{d^2} <  \frac{1}{2} \mcom
    \end{align*}
    where the last inequality follows for some $m < c d^2$ for small enough $c$.
    Thus, $u = -1 + \frac{\eta^T A^{-1} \eta}{d} \in [-1, -1/2]$.

    We defer the proof for $s$ to \Cref{claim: bound-for-s} in the appendix. With the bounds on $r$, $s$ and $u$, we immediatley get $s^2 - ru \geq \Omega(\frac{m}{d})$.

    To prove that $M$ is invertible, let us first recall that we write $M = A + UCU^T$ where $A$ is defined in \Cref{eq:M-decomposition} and $U = V^T = \frac{1}{\sqrt{d}} \begin{bmatrix} 1_m & \eta \end{bmatrix} \in \R^{m \times 2}$ and
    $C = \begin{bmatrix} 1 & 1 \\ 1 & 0 \end{bmatrix}$.

    By \Cref{lem:A-eigenvalues}, $A$ is invertible. 
    Then by \Cref{fact:invertibility}, we know that $M$ is invertible if and only if $C^{-1} + U^T A^{-1} U \coloneqq \begin{bmatrix}
        r & s \\ s & u
    \end{bmatrix}$ (see \Cref{eq:russ}) is invertible, which is equivalent to $ru - s^2 \neq 0$.
    Thus, $s^2 - ru \geq \Omega(\frac{m}{d})$ suffices to conclude that $M$ is invertible.
\end{proof}

\subsection{Finishing the proof of \texorpdfstring{\Cref{thm:main-thm}}{Theorem \ref{thm:main-thm}}}
\label{sec:finish-main-thm}

The final piece of proving \Cref{thm:main-thm} is to show that $\cR = \sum_{i=1}^m w_i v_i v_i^T$ has spectral norm bounded by 1, which immediately implies that the candidate matrix $\Lambda = I_d - \cR \succeq 0$.

\begin{lemma}[$\calR$ is bounded]\label{lem:R-norm-bound}
 There exists some absolute constant $C_R$ such that for $m \leq \frac{d^2}{C_R}$, \[ 
 \Norm{\cR}_{\op} \leq \frac{1}{2} \mper
 \]
\end{lemma}

The proof is deferred to \Cref{sec:R-decomposition}.
In particular, we will write an expanded expression of $M^{-1}$ and obtain a decomposition of $\cR$ (\Cref{prop:R-decomposition}).
Then, in \Cref{sec:R-analysis}, we prove tight spectral norm bounds for matrices in the decomposition, which then completes the proof of \Cref{lem:R-norm-bound}.

Combining \Cref{lem:M-invertible} and \cref{lem:R-norm-bound}  we can finish the proof of \Cref{thm:main-thm}.

\begin{proof}[Proof of \Cref{thm:main-thm}]
    The matrix $M$ (recall \Cref{eq:M-matrix}) is invertible due to \Cref{lem:M-invertible}, thus our candidate matrix $\Lambda = I_d - \cR$ matrix defined in \Cref{def:candidate-matrix} is well-defined.
    Furthermore, by \Cref{lem:R-norm-bound} we have that $\|\cR\|_{\op} <1 $.
    This proves that $\Lambda \succ 0$. 
\end{proof}
\section{Machinery for tight norm bounds of graph matrices}
\label{sec:overview-norm-bounds}
One of the main technical contributions of this paper is providing tight spectral norm bounds (up to constants per vertex/edge) for \emph{structured random matrices with correlated entries} (a.k.a.\ graph matrices).
We note that prior to this work, most known norm bounds for such matrices are only tight up to some logarithmic factors~\cite{AMP16}, while not much is known in terms of precise bounds without log factors except for several specific cases (see e.g.~\cite{Tao12}).

\subsection{Preliminaries}
\label{sec:graphical-matrices}

We first give a lightweight introduction to the theory of graph matrices.
For interested readers who seek a thorough introduction or a more formal treatment, we refer them to its origin in a sequence of works in Sum-of-Squares lower bounds \cite{BHK+19, AMP16}.
We will follow the notations used in \cite{AMP16}.
Throughout this section, we assume that there is an underlying (random) input matrix $G$ and a Fourier basis $\{\chi_t\}_{t\in\N}$.

We first define \emph{shapes}, which are representations of structured matrices whose entries depend on $G$.

\begin{definition}[Shape] A shape $\tau$ is a tuple $(V(\tau), U_\tau, V_\tau, E(\tau))$ associated with a (multi) graph $(V(\tau), E(\tau))$.
Each vertex in $V(\tau)$ is associated with a vertex-type that indicates the range of the labels for the particular vertex.
Each edge $e\in E(\tau)$ is also associated with a Fourier index $t(e) \in \N$.
Moreover, we have $U_\tau, V_\tau \subseteq V(\tau)$ as the left and right boundary of the shape.
\end{definition}

We remind the reader that $V_\tau$ should be distinguished from $V(\tau)$, where $V_\tau$ is the right boundary set, while $V(\tau)$ is the set of all vertices in the graph.

\Cref{fig:M_alpha_beta} show the shapes for matrices $M_\al$ and $M_\beta$ defined in \Cref{prop:M-decomposition}.
For these shapes, there are two vertex-types (square and circle).
The two ovals in each shape indicate the left and right boundaries $U_\tau$ and $V_\tau$.

We next describe how to associate a shape to a matrix (given the underlying matrix $G$).

\begin{figure}[ht!]
    \centering
    \begin{subfigure}[b]{0.25\textwidth}
        \includegraphics[width=\textwidth]{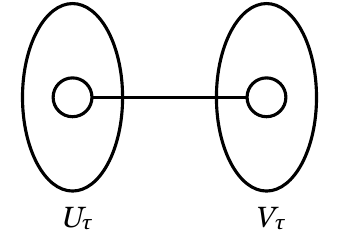}
        \caption{$\mathsf{GOE}$, zero diagonal.}
        \label{fig:GOE}
    \end{subfigure}
    \quad
    \begin{subfigure}[b]{0.32\textwidth}
        \includegraphics[width=\textwidth]{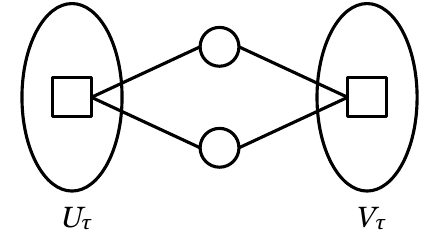}
        \caption{$M_{\alpha}$.}
        \label{fig:M_alpha}
    \end{subfigure}
    \quad
    \begin{subfigure}[b]{0.32\textwidth}
        \includegraphics[width=\textwidth]{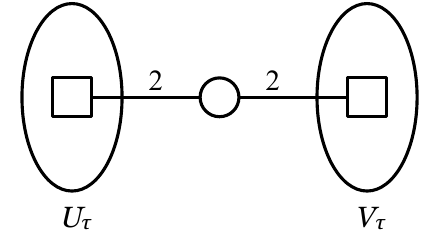}
        \caption{$M_{\beta}$.}
        \label{fig:M_beta}
    \end{subfigure}
    \captionsetup{width=.9\linewidth}
    \caption{Graph matrix representation of a $d\times d$ $\mathsf{GOE}$ matrix with zero diagonal, and the $m \times m$ matrices $M_{\alpha}$ and $M_{\beta}$ as defined in \Cref{prop:M-decomposition}.
    Square vertices take labels in $[m]$ and circle vertices take labels in $[d]$.
    The two ovals indicate the left and right boundaries $U_\tau, V_\tau$ of the shapes.
    If an edge $e$ is not labeled with an index, then $t(e) = 1$ by default.}
    \label{fig:M_alpha_beta}
\end{figure}


\begin{definition}[Mapping of a shape] Given a shape $\tau$, we call a function $\sigma : V(\tau) \to \N$ a mapping of the shape if \begin{enumerate}
    \item $\sigma$ assigns a label for each vertex according to its specified vertex-type;
    \item $\sigma$ is an injective mapping for vertices of the same type.
\end{enumerate}
    
\end{definition}

\begin{definition}[Graph matrix for shape] \label{def:graph-matrix-for-shape}
    Given a shape $\tau$, we define its graphical matrix $M_{\tau}$ to be the matrix indexed by all possible boundary labelings of $S, T$, and for each of its entry, we define
    \[ 
    M_{\tau}[S, T] = \sum_{\substack{\sigma: V(\tau) \to \N \\ \sigma(U_\tau) = S,\ \sigma(V_\tau) = T }} \prod_{e\in E(\tau)}\chi_{t(e)}(G[\sigma(e)]) \mper
    \]
\end{definition}

Observe that for each entry $M_{\tau}[S,T]$, since $\sigma$ must map $U_\tau$ and $V_\tau$ to $S$ and $T$,
$M_{\tau}[S,T]$ is simply a sum over labelings of the ``middle'' vertices $V(\tau) \setminus (U_\tau \cup V_\tau)$.
Take \Cref{fig:M_alpha_beta} for example. Suppose $G\in \R^{m \times d}$ and square and circle vertices take labels in $[m]$ and $[d]$ respectively, then we can write out the entries of the matrix: for $i \neq j\in [m]$,
\begin{align*}
    M_{\al}[i,j] &= \sum_{a\neq b\in [d]} \chi_1(G[i,a]) \cdot \chi_1(G[i,b]) \cdot \chi_1(G[j,a]) \cdot \chi_1(G[j,b]) \mcom \\
    M_{\beta}[i,j] &= \sum_{a\in [d]} \chi_2(G[i,a]) \cdot \chi_2(G[j,a]) \mper
\end{align*}

Note also that since $\sigma$ must be injective for vertices of the same type and $U_\tau \neq V_\tau$ in both examples, there is no mapping such that $\sigma(U_\tau) = \sigma(V_\tau)$.
Thus, by \Cref{def:graph-matrix-for-shape}, both matrices have zeros on the diagonal.


\paragraph{Adaptation to our setting} The above is a general introduction for graph matrices.
In this work, we specialize to the following setting:
\begin{itemize}
    \item $G \in \R^{m\times d}$ is a random Gaussian matrix whose rows are $v_1,\dots, v_m \sim \calN(0, \frac{1}{d} I_d)$.
    \item The Fourier characters $\{\chi_t\}_{t\in\N}$ are the (scaled) Hermite polynomials.
    \item For all graph matrices that arise in our analysis,
    \begin{itemize}
        \item $|S| = |T| = 1$,
        \item There are two vertex-types: square vertices take labels in $[m]$ and circle vertices take labels in $[d]$.
    \end{itemize}
\end{itemize}

\begin{remark}
	For our technical analysis, we may also employ this machinery on a broader range of graph matrices for shape in which we relax the local injectivity condition within each block. That said, for illustration purpose, it suffices to consider the vanilla setting of graph matrix.
\end{remark}
\begin{definition}[$D_V$ size constraint]
For each graph matrix $\tau$ considered in this work, let $D_V$ be the size constraint such that $|V(\tau)|\leq D_V$. 
\end{definition}
For concreteness, we will take $D_V = \polylog(d)$ throughout this work.

%
%
%
\paragraph{Trace moment method}

For all our norm bounds, we will use the trace moment method: for any graph matrix $M_{\tau}$ with underlying random matrix $G$ and any $q \in \N$,
\begin{align*}
    \E \|M_{\tau} \|_{\op}^{2q} &\leq \E \tr\Paren{(M_{\tau} M_{\tau}^T)^{q}}
    = \E \sum_{\substack{S_1, T_1, S_2, T_2,\dots S_{q-1}, T_{q-1}: \\ \text{boundaries}} } M_{\tau}[S_1,T_1] M_{\tau}^T[T_1,S_2] \cdots M_{\tau}^T[T_{q-1}, S_1] \mper
\end{align*}
where the expectation is taken over $G$.

Notice that the summation is over \emph{closed walks} across the boundaries: $S_1 \to T_1 \to S_2 \to T_2\to \dots \to S_1$, where $S_1, T_1, \dots$ are boundary labelings of $M_{\tau}$. In particular, the walk is consist of $2q$-steps of ``block walk'', with the $(2t-1)$-th step across a block described by $M_{\tau}$ and the $(2t)$-th step across a block described by $M_{\tau}^T$.

The crucial observation is that after taking expectation, all closed walks must walk on each labeled edge (i.e., Fourier character) an \emph{even} number of times, since all odd moments of the Fourier characters are zero. Therefore, bounding the matrix norm is reduced to bounding the contribution of all such walks.
\begin{align*}
    \E \|M_{\tau}\|_{\op}^{2q}
    \leq \sum_{\substack{\calP: \text{ closed walk} }} \prod_{e\in E(\calP)}  \E\Brac{\chi_{t(e)}(G[e])^{\mul_\calP(e)}} \mcom
    \numberthis \label{eq:M-beta-trace}
\end{align*}

where $E(\calP)$ denotes the set of labeled edges used by the walk $\calP$, $\mul_\calP(e)$ denotes the number of times $e$ appears in the walk, and $t(e)$ denotes the Fourier index (with slight abuse of notation).

\begin{remark}[Labeled vertex/edge] \label{rem:labeled-vertices-edges}
    We remind the reader not to confuse vertices/edges in the walk with vertices/edges in the shape.
    The vertices in a walk are ``labeled'' by elements in $[m]$ or $[d]$ (depending on the vertex-type).
    Similarly, each edge $e\in E(\calP)$ in a walk is labeled by an element in $[m] \times [d]$.
    We will use the terms ``labeled vertex'' and ``labeled edge'' unless it is clear from context.
\end{remark}

\subsection{Global bounds via a local analysis}
\label{sec:overview-local-analysis}


Observe that \Cref{eq:M-beta-trace} is a weighted sum of closed walks of length $2q$.
To obtain an upper bound, the standard approach is to specify an efficient \emph{encoding scheme} that uniquely identifies each closed walk, and then upper bound the total number of such encodings.

We begin by defining a step-labeling --- a categorization of each step in the closed walk.

\begin{definition}[Step-labeling]
\label{def:step-labeling}
For each step throughout the walk, we assign it the following label,\begin{enumerate}
		\item $F$ (a fresh step): it uses a new labeled edge making the first appearance and leads to a destination not seen before;
		\item $S$ (a surprise step): it uses a new labeled edge to arrive at a vertex previously visited in the walk;
		\item $H$ (a high-mul step): it uses a labeled edge that appears before, and the edge is making a middle appearance (i.e., it will appear again in the subsequent walk);
		\item $R$ (a return step): it uses a labeled edge that appears before, and the edge is making its last appearance.
	\end{enumerate}
	Analogously, for any shape $\tau$ , we call $\calL_\tau :E(\tau)\rightarrow\{F,R,S,H\}
$ a \emph{step-labeling} of the block. The subscript $\tau$ is ignored when it is clear.
\end{definition}

We note that the terms ``fresh'', ``high-mul'' and ``return'' are adopted from the GOE matrix analysis in \cite{Tao12}.
Next, to obtain a final bound for \Cref{eq:M-beta-trace}, we consider two \emph{factors} for each step (which depend on the step-label):
\begin{enumerate}
    \item \textbf{Vertex factor}: a combinatorial factor that specifies the destination of the step;
    \item \textbf{Edge factor}: an analytical factor from the edge which accounts for the $\E[\chi_{t(e)}(G[e])^{\mul(e)}]$ term in \Cref{eq:M-beta-trace}.
\end{enumerate}

For example, a vertex factor for an $F$ step to a circle vertex can be $d$, an upper bound on the number of possible destinations.
One can think of vertex factors as the information needed for a \emph{decoder} to complete a closed walk.
Essentially, the step-labeling and appropriate vertex factors should uniquely identify a closed walk, and combined with edge factors, we can obtain an upper bound for \Cref{eq:M-beta-trace}.

We note that the approach stated above is a \emph{global} encoding scheme.
One may proceed via a global analysis --- carefully bounding the number of step-labelings allowed (e.g., using the fact that the $F$ and $R$ steps must form a Dyck word~\cite{Tao12}), and then combining all vertex and edge factors to obtain a final bound.
However, to get tight norm bounds for complicated graph matrices (like $M_\alpha$), the global analysis becomes unwieldy.

\paragraph{Local analysis}
One of our main insights is to use a \emph{local} analysis.
We now give a high-level overview of our strategy while deferring the specific details of our vertex/edge factor assignment scheme to subsequent sections.
Recall that a closed walk consists of ``block-steps'' described by the shape $\tau$.
Thus, we treat each walk as a ``block walk'' and bound the contributions of a walk block by block.
This prompts us to bound the contribution of the walk at a given block-step to the final trace in \cref{eq:M-beta-trace} by
\[ 
\vtxcost \cdot \edgeval \leq B_q(\tau)
\]
where $B_q(\tau)$ is some desired upper bound that depends on the vertex/edge factor assignment scheme.
We define it formally in the following.


\begin{definition}[Block value function] \label{def:block-value-function}
Fix $q\in \N$ and a shape $\tau$. For any vertex/edge factor assignment scheme, we call $B_q(\tau)$ a valid block-value function for $\tau$ of the given scheme if   \[ 
\E\left[\tr \left((M_{\tau} M_{\tau}^T)^{q} \right)\right] \leq (\text{matrix dimension}) \cdot B_q(\tau)^{2q} \mcom
\]
and for each block-step $\mathsf{BlockStep}_i$ throughout the walk, \[ 
 \vtxcost(\mathsf{BlockStep}_i) \cdot \edgeval(\mathsf{BlockStep}_i) \leq B_q(\tau)\,.
\]
\end{definition}
We point out that the block-value function $B$ should be considered as a function of both the shape $\tau$ and the length of the walk $q$ (we will drop the subscript when it is clear throughout this work), and it also depends on the assignment scheme.
Thus, our task is to find a vertex/edge factor assignment scheme such that $B_q(\tau)$ is as small as possible.
Moreover, the $\text{matrix dimension}$, which is at most $\poly(d)$ in our case, is the factor that comes up in the start of the walk to specify the original vertex, and can be ignored as it is ultimately an $1+o(1)$ factor once we take a long enough walk.

Given \Cref{def:block-value-function}, the norm bound follows immediately from Markov's inequality.

\begin{proposition}  \label{prop:norm-bound-from-B}
    Let $M_{\tau}$ be a graph matrix with dimension $\poly(d)$, and let $q \geq \Omega(\log^2 d)$.
    Suppose $B_q(\tau)$ is a valid block-value function,
    Then, with probability $1-2^{-q/\log d}$,
    \begin{align*}
        \|M_{\tau}\|_{\op} \leq (1+o_d(1)) \cdot B_q(\tau) \mper
    \end{align*}
\end{proposition}
\begin{proof}
    We apply Markov's inequality: for any $\eps > 0$,
    \begin{align*}
        \Pr\Brac{\|M_{\tau}\|_{\op} > (1+\eps) B_q(\tau) } 
        &\leq \Pr\Brac{\tr\Paren{(M_{\tau} M_{\tau}^T)^{q}} > (1+\eps)^{2q} B_q(\tau)^{2q} }  \\
        &\leq (1+\eps)^{-2q} \poly(d) \leq e^{-2\eps q} \poly(d)
    \end{align*}
    since $\E[\tr((M_{\tau}M_{\tau}^T)^q)] \leq \poly(d) \cdot B_q(\tau)^{2q}$ by \Cref{def:block-value-function}.
    Setting $\eps = \frac{1}{\log d}$ and $q \geq \Omega(\log^2 d)$, we have that the probability is at most $2^{-q / \log d}$.
    Thus, we can conclude that $\|M_{\tau}\|_{\op} \leq (1+o_d(1))\cdot B_q(\tau)$ with probability $1 - 2^{-q/
    \log d}$.
\end{proof}

The next proposition shows that we can easily obtain a valid $B_q(\tau)$ once we have an appropriate factor assignment scheme.

\begin{proposition} \label{prop:sum-of-labelings}
    For any graph matrix $M_\tau$ and any valid factor assignment scheme,
    \begin{align*}
		B_q(\tau) = \sum_{\calL : \text{ step-labelings for } E(\tau)} \vtxcost(\calL) \cdot \edgeval(\calL)
	\end{align*}
    is a valid block-value function for $\tau$.
\end{proposition}
\begin{proof}
    It is clear that the second requirement in \Cref{def:block-value-function} is satisfied.
    For the first requirement, observe that the trace can be bounded by the matrix dimension (specifying the start of the walk) times
    \begin{equation*}
        \sum_{\substack{\calL_1,\dots, \calL_{2q}: \\ \text{step-labelings for $E(\tau)$}}} \prod_{i=1}^{2q} \vtxcost(\calL_i) \cdot \edgeval(\calL_i)
        \leq \Paren{ \sum_{\calL:\text{step-labelings for } E(\tau)} \vtxcost(\calL) \cdot \edgeval(\calL) }^{2q} \mper
        \qedhere
    \end{equation*}
\end{proof}

With this set-up, the main task is then to find an appropriate vertex/edge factor assignment scheme and obtain a good upper bound on $B_q(\tau)$.

 \subsection{Vertex factor assignment scheme}
 \label{sec:vertex-factor-scheme}
 
We now proceed to bound the vertex factors for each step-label.
We note that in this section, ``vertices'' refer to ``labeled vertices'' in the walk (having labels in $[m]$ or $[d]$; recall \Cref{rem:labeled-vertices-edges}).
First, we define the weight of a square (resp.\ circle) vertex to be $m$ (resp.\ $d$), since we need an element in $[m]$ (resp.\ $[d]$) to specify which vertex to go to in the walk.

We first show a ``naive'' vertex factor assignment scheme.
In the following scheme, we use a \emph{potential unforced return} factor, denoted $\pur$, to specify the destination of any $R$ step.
We will defer the specific details of $\pur$ to \Cref{sec:resolving-confusion}.
\begin{mdframed}[frametitle = {Vanilla vertex factor assignment scheme} ]
\begin{enumerate}
    \item For each vertex $i$ that first appears via an $F$ step, a label in $\wt(i)$ is required;
    \item For each vertex $i$ that appears beyond the first time:
    \begin{itemize}
        \item If it is arrived via an $R$ step, the destination may need to be specified, and this is captured by the $\pur$ factor. 
        \item If it is \emph{not} arrived via an $R$ step, then it must be an $S$ or $H$ step.
        A vertex cost in $2q\cdot D_V$ is sufficient to identify the destination, where we recall $2q$ is the length of our walk, and $D_V$ the size upper bound of each block.
    \end{itemize}
\end{enumerate}
\end{mdframed}

The first thing to check is that this scheme combined with an step-labeling uniquely identifies a closed walk (given the start of the walk).
This is immediate for $F$ and $R$ steps by definition.
For $S$ and $H$ steps, since the destination is visited before in the walk, $2q \cdot D_V$ is sufficient as it is an upper bound on the number of vertices in the walk.

A potential complication with analyzing the above assignment scheme directly is that it exhibits a significant difference in the vertex factors.
For example, consider a vertex that appears only twice in the walk on a tree. Its first appearance requires a label in $[n]$, while its subsequent appearance does not require any cost if it is reached using an $R$ step because backtracking from a tree is fixed (since there is only one parent).
This disparity can result in a very loose upper bound for the trace when applying \Cref{prop:sum-of-labelings}; in fact, the norm bound for $M_\tau$ obtained in this manner is equivalent to using the naive row-sum bound.

\paragraph{Redistribution}
One of our main technical insights is to split the factors such that both first and last appearance contributes a factor of comparable magnitude; we call this \emph{redistribution}.

We first formally define ``appearance'' in a block-step to clarify our terminology,
\begin{definition}[Vertex appearance in block-step]  \label{def:vertex-appearance}
    Each labeled vertex appearance can be ``first'', ``middle'' and ``last''.
    Moreover, each vertex on the block-step boundary ($U_\tau$ or $V_\tau$) appears in both adjacent blocks.
\end{definition}

For example, suppose a vertex first appears in the right-boundary of block $i$ and last appears in the left-boundary of block $j$, then it will make middle appearances in the left-boundary of block $i+1$ and right-boundary of block $j-1$ as well.

We are now ready to introduce the following vertex-factor assignment scheme with redistribution that assigns vertex-factor to each vertex's appearance to handle the disparity.
\begin{mdframed}[frametitle = {Vertex factor assignment scheme with redistribution } ]
\begin{enumerate}
    \item For each vertex $i$ that makes its first appearance, assign a cost of $\sqrt{\wt(i)}$;
    \item For any vertex's middle appearance, if it is not arrived at via an $R$ step, assign a cost of $2q\cdot D_V $  (where we recall $2q$ is the length of our walk, and $D_V$ the size constraint of each block);
    \item For any vertex's middle appearance, if it is arrived at via an $R$ step, its cost is captured by $\pur$;
    \item For each vertex $i$ that makes its last appearance, assign a cost of $\sqrt{\wt(i)}$ that serves as a \emph{backpay}.
\end{enumerate}
\end{mdframed}

%

\paragraph{Deducing vertex factor from local step-labeling}
As presented, the vertex factor assignment scheme requires knowing which vertex is making first/middle/last appearance. We further show that the vertex appearances, or more accurately, an upper bound of the vertex factors, can be deduced by a given step-labeling of the block. Fix traversal direction from $U$ to $V$,
\begin{mdframed}[frametitle = {Localized vertex factor assignment from step-labeling }]
	\begin{enumerate}
		\item For any vertex $v$ that is on the left-boundary $U$, it cannot be making the first appearance since it necessarily appears in the previous block;
		\item For any vertex $v$ that is on the right-boundary $V$, it cannot be making the last appearance since it necessarily appears in the subsequent block;
		\item For any vertex $v$ reached via some $S/R/H$ step, it cannot be making its first appearance;
		\item For any vertex $v$ that incident to some $F/S/H$ step, it cannot be making its last appearance since the edge necessarily appears again.
	\end{enumerate}
\end{mdframed}
The first two points are due to \Cref{def:vertex-appearance}. The last point is because each labeled edge (i.e., Fourier character) must be traversed by an $R$ step to close it.

\subsection{Bounding edge-factors}
\label{sec:edge-factor-scheme}

To bound the contribution of the walks, we need to consider factors coming from the edges traversed by the walk.
Recall from \Cref{eq:M-beta-trace} that each edge $e$ in a closed walk $P$ gets a factor $\E[\chi_{t(e)}^{\mul_P(e)}]$, where $t(e)$ is the Fourier index associated with the edge.

In our case, the Fourier characters are the scaled Hermite polynomials.
Recall that we assume that our vectors are sampled as $v_i\sim \calN(0, \frac{1}{d}I_d)$.
Thus, we define the polynomials $\{h_t\}_{t\in\N}$ such that they are orthogonal and $\E_{x\sim \calN(0,1/d)}[h_t(x)^2] = t! \cdot d^{-t}$.
Specifically,
\begin{enumerate}
    \item $h_1(x) = x$,
    \item $h_2(x) = x^2 - \frac{1}{d}$.
\end{enumerate}
We first state the following bound on the moments of $h_t$, which follows directly from standard bounds on the moments of Hermite polynomials:

\begin{fact}[Moments of Hermite polynomials]  \label{fact:hermite-moments}
    Let $d\in \N$. For any $t\in \N$ and even $k \in \N$,
    \[
        \E_{x\sim \calN(0,1/d)} \Brac{h_t(x)^k} \leq \frac{1}{d^{kt/2}} (k-1)^{kt/2} (t!)^{k/2} \leq (t!)^{k/2} \Paren{\frac{k}{d}}^{kt/2} \mper
    \]
\end{fact}

For now, we consider matrices that either contain only $h_1$ or only $h_2$ edges (the edge factors for graph matrices with ``mixed'' edges will be handled in \Cref{sec:further-setup-edge-factors}).
The following is our edge-factor assignment scheme to account for contributions from the Fourier characters.


\begin{mdframed}[frametitle = {Edge-factor assignment scheme } ]
	For an $h_1$ edge,\begin{enumerate}
		\item $F/S$: assign a factor of $\frac{1}{\sqrt{d}}$ for its first appearance;
		\item $H$: assign a factor of $\frac{2q}{\sqrt{d}}$ for its middle appearance;
		\item $R$: assign a factor of $\frac{1}{\sqrt{d}}$ for its last appearance.
	\end{enumerate}
	For an $h_2$ edge, \begin{enumerate}
		\item $F/S$: assign a factor of $\frac{\sqrt{2}}{d}$ for its first appearance (alternatively, we can view a single $h_2$ edge as two edge-copies of $h_1$ and assign each a factor of $\frac{\sqrt{2}}{\sqrt{d}}$, which is a valid upper bound);
		\item $H$: assign a factor of $\frac{8q^2}{d}$ for its middle appearance;
		\item  $R$: assign a factor of $\frac{\sqrt{2}}{d}$ for its last appearance (alternatively, we can view a single $h_2$ edge as two edge-copies of $h_1$ and assign each a factor of $\frac{\sqrt{2}}{\sqrt{d}}$ which is a valid upper bound).
	\end{enumerate}
\end{mdframed}

\begin{proposition} \label{prop:edgeval-pure-edges}
    The above scheme correctly accounts for the edge factors from $h_1$ and $h_2$ edges.
\end{proposition}
\begin{proof}
    If an edge has multiplicity $2$ then it must be traversed by one $F/S$ step and one $R$ step.
    \begin{itemize}
        \item If it is an $h_1$ edge, then the scheme assigns a factor $\frac{1}{d}$, which equals $\E_{x\sim \calN(0,1/d)}[h_1(x)^2]$.
        \item If it is an $h_2$ edge, then the scheme assigns a factor $\frac{2}{d^2}$, which equals $\E_{x\sim \calN(0,1/d)}[h_2(x)^2]$.
    \end{itemize}
    For an edge with multiplicity $k > 2$, it must be traversed by one $F$ step (including $S$), one $R$ step and $k-2$ $H$ steps. Moreover, since $k$ is even and $2q$ is the length of the walk, we have $4 \leq k \leq 2q$.
    \begin{itemize}
        \item If it is an $h_1$ edge, then the scheme assigns a factor $\frac{1}{d} \cdot (\frac{2q}{\sqrt{d}})^{k-2} \geq d^{-k/2} (2q)^{k/2} \geq (\frac{k}{d})^{k/2}$.
        By \Cref{fact:hermite-moments}, it is an upper bound on $\E_{x\sim \calN(0,1/d)}[h_1(x)^k]$.
        \item If it is an $h_2$ edge, then the scheme assigns a factor $\frac{2}{d^2} \cdot (\frac{8q^2}{d})^{k-2} \geq d^{-k} 2^{k/2} (2q)^{k} \geq 2^{k/2} (\frac{k}{d})^{k}$.
        By \Cref{fact:hermite-moments}, it is an upper bound on $\E_{x\sim \calN(0,1/d)}[h_2(x)^k]$.
    \end{itemize}
    This shows that the edge factor assignment scheme above is correct.
\end{proof}

\subsection{Bounding return cost (\texorpdfstring{$\pur$}{Pur} factors)}
\label{sec:resolving-confusion}

In our vertex factor assignment scheme described in \Cref{sec:vertex-factor-scheme}, we use a \emph{potential unforced return} factor, denoted $\pur$, to specify the destination of any return ($R$) step.
Note that the term ``unforced return'' is adopted from \cite{Tao12} as well.
In this section, we complete the bound of vertex factors by bounding the $\pur$ factor.

For starters, we will define a potential function for each vertex at time $t$, which measures the number of returns $R$ pushed out from the particular vertex by time $t$ that may require a label in $2q \cdot D_V$. Notice that a label in $2q \cdot D_V$ is sufficient for any destination vertex arrived via an $R$ step because the vertex appears before; however, this may be a loose bound.

We observe the following: a label in $2q\cdot D_V$ may be spared if the vertex is incident to only one unclosed $F/S$ edge; we call this a \emph{forced} return.
Formally, we define a return step as \emph{unforced} if it does not fall into the above categories,
\begin{definition}[Unforced return]
We call a return ($R$) step an \emph{unforced return} if the source vertex is incident to more than $1$ (or $2$ in the case of a square vertex) unclosed edge.
\end{definition}
We now proceed to formalize the above two observations by introducing a potential function to help us bound the number of unforced returns from any given vertex throughout the walk. The number of unforced returns throughout the walk would then be immediately given once we sum over all vertices in the walk.
\begin{definition}[Potential-unforced-return factor $\pur$]
\label{def:overview-pur}
For any time $t$ and vertex $v$, let $\pur_t(v)$ be defined as the number of potential unforced return from $v$  throughout the walk until time $t$.
\end{definition}


\subsubsection{\texorpdfstring{$\pur$}{Pur} bound for circle vertices}

In our setting, each circle vertex pushes out at most $1$ edge during the walk, analogous to the case of typical adjacency matrix.  This serves as a starting point for our $\pur$ bound for circle vertices.


\begin{lemma}[Bounding $\pur_t$ for circle vertices] \label{lem:pur-circle}
    For any time $t$, suppose the walker is currently at a circle vertex $v$, then
    \begin{align*}
        \pur_t(v) &\leq \# \text{(R steps closed from $v$)} + \#\text{(unclosed edges incident to $v$ at time $t$)} - 1 \\
        &\leq 2\cdot s_t(v) +  h_t(v) \mcom
    \end{align*}
    where we define the following counter functions:
    \begin{enumerate}
        \item $s_t(v)$ is the number of $S$ steps arriving at $v$ by time $t$;
        \item $h_t(v)$ is the number of $H$ steps arriving at $v$ by time $t$.
    \end{enumerate}
    \end{lemma}
\begin{proof}
    We first prove the first inequality.
    The $R$ steps closed from $v$ may all be unforced returns, and the unclosed edges incident to $v$ may be closed by unforced returns in the future.
    Note that we have a $-1$ in the above bound because for each vertex we may by default assume the return is using a particular edge, hence at each time we know there is an edge presumed-to-be forced.

    We prove the second inequality by induction.
    Define $P_t(v) \coloneqq \# \text{(R steps closed from $v$)} + \#\text{(unclosed edges incident to $v$ at time $t$)} - 1$ for convenience.
    At the time when $v$ is first created by an $F$ step, $P_t(v) = 0$ (1 open edge minus $1$) and $s_t(v) = h_t(v) = 0$.

    At time $t$, suppose the last time $v$ was visited was at time $t' < t$, and suppose that the inequality holds true for $t'$.
    Note that at time $t'+1$, $P_{t'+1}(v) = P_{t'}(v)+1$ if a new edge was created by an $F$ or $N$ step leaving $v$, otherwise $P_{t'+1}(v) = P_{t'}(v)$ (for $R$ step it adds 1 to the number of closed edges closed from $v$, but decreases 1 open edge).
    On the other hand, $s_{t'}(v)$ and $h_{t'}(v)$ remain the same (we don't count out-going steps for $s_t(v), h_t(v)$).
    
    When we reach $v$ at time $t$, we case on the type of steps:
    \begin{itemize}
        \item Arriving by an $R$ step: the edge is now closed, but the $R$ step was not from $v$. So $P_t(v) = P_{t'+1}(v) - 1 \leq P_{t'}(v)$, while $s_t(v) = s_{t'}(v)$ and $h_t(v) = h_{t'}(v)$.
        
        \item Arriving by an $S$ step: the edge is new, so $P_t(v) = P_{t'+1}(v)+1 \leq P_{t'}(v) + 2$, and we have $s_t(v) = s_{t'}(v) +1$.

        \item Arriving by an $H$ step: $P_t(v) = P_{t'+1}(v) \leq P_{t'}(v) + 1$, and $h_t(v) = h_{t'}(v)+1$.
    \end{itemize}
    In all three cases, assuming $P_{t'}(v) \leq 2 \cdot s_{t'}(v) + h_{t'}(v)$, we have $P_t(v) \leq 2\cdot s_t(v) + h_t(v)$, completing the induction.
\end{proof}

\subsubsection{\texorpdfstring{$\pur$}{Pur} bound for square vertices}

The argument of \Cref{lem:pur-circle} does not apply well for vertices incident to multiple edges in a single step. In particular, this may happen for square vertices in $M_\al$ as each is arrived via $2$ edges and each pushes out $2$ edges (recall \Cref{fig:M_alpha_beta}).
This is not an issue for $M_{\beta}$, but we will treat square vertices in $M_{\beta}$ the same way to unify the analysis; in the context of $\pur$ for square vertices, one may think of $M_{\beta}$ as collapsing the two circle vertices in $M_{\alpha}$.

To handle this issue, we observe that it suffices for us to pay an extra cost of $[2]$ for each square vertex, which would allow us to further presume $2$ edges being forced. We then generalize the prior argument to capture this change.

\begin{lemma}[Bounding $\pur_t$ for square vertices] \label{lem:pur-square}
    For any time $t$, suppose the walker is currently at a square vertex $v$, then
    \begin{align*}
        \pur_t(v) &\leq \# \text{(R steps closed from $v$)} + \#\text{(unclosed edges incident to $v$ at time $t$)} - 2 \\
        &\leq 2(s_t(v) +  h_t(v)) \mper
    \end{align*}
    where $s_t(v)$ and $h_t(v)$ are the number of $S$ and $H$ steps arriving at $v$ by time $t$, respectively.
\end{lemma}
\begin{proof}
We prove this by induction. Note that this is immediate for the base case when $v$ first appears since a square vertex is incident to $2$ edges.
Define $P_t(v) \coloneqq \# \text{(R steps closed from $v$)} + \#\text{(unclosed edges incident to $v$ at time $t$)} - 2$ for convenience.
Suppose the inequality is true at time $t'$, and assume vertex $v$ appears again at time $t$. The departure at time $t'+1$ from $v$ may open up at most $2$ edges, hence $P_{t'+1}(v) \leq P_{t'}(v) + 2$.

When we reach $v$ at time $t$ (via 2 edges), we case on the type of steps:
\begin{itemize}
    \item Arriving by two $R$ steps: the two edges closed by the $R$ steps are not closed from $v$.
    So $P_t(v) = P_{t'+1}-2 \leq P_{t'}(v)$, while $s_t(v) = s_{t'}(v)$ and $h_t(v) = h_{t'}(v)$.

    \item Arriving by one $S/H$ and one $R$ step: in this case, $P_t(v) = P_{t'+1}(v) \leq P_{t'}(v)+2$ and $s_t(v) + h_t(v) = s_{t'}(v) + h_{t'}(v) + 1$.

    \item Arriving by two $S/H$ steps: in this case, $P_{t}(v) = P_{t'+1}(v) + 2 \leq P_{t'}(v) + 4$, whereas $s_t(v) + h_t(v) = s_{t'}(v) + h_{t'}(v) + 2$.    
\end{itemize}
In all three cases, we have $P_{t}(v) \leq 2(s_{t}(v) + h_t(v))$, completing the induction.
\end{proof}

\begin{corollary} \label{cor:pur-square}
For each surprise/high-mul step, it suffices for us to assign $2$ $\pur$ factors, which is a cost of $(2q\cdot D_V)^{2}$ so that each $\pur$ factor throughout the walk is assigned.
Moreover, for $M_\al$, we pay a cost of $2$ for any $R$ step leaving a square vertex so that we can presume $2$ edges being forced in \Cref{lem:pur-square}.
\end{corollary}

\subsection{Wrapping up with examples}
\label{sec:wrapping-up}

Recall \Cref{prop:sum-of-labelings} that for a graph matrix of shape $\tau$,
\begin{align*}
    B_q(\tau) = \sum_{\calL:\text{ step-labelings for } E(\tau)} \vtxcost(\calL) \cdot \edgeval(\calL)
    \numberthis  \label{eq:Bq-restated}
\end{align*}
is a valid block-value function for $\tau$ (\Cref{def:block-value-function}).
Moreover, by \Cref{prop:norm-bound-from-B}, we can take $q \geq d^{\eps}$ and conclude that with probability $1- 2^{-d^{\eps}}$,
\begin{equation*}
    \|M_{\tau}\|_{\op} \leq (1+o(1)) \cdot B_q(\tau) \mper
\end{equation*}

For each given shape, it suffices for us to bound the block-value for each edge-labeling. And we demonstrate how this may be readily done given the above bounds.

\subsubsection{Warm-up: tight bound for GOE}
As a warm-up, we first see how the above framework allows us to readily deduce a tight norm bound for $G \sim \GOE(0,\frac{1}{d})$, where $G$ is a $d\times d$ symmetric matrix with each (off-diagonal) entry sampled from $\calN(0,\frac{1}{d})$.
It is well-known that the correct norm of $G$ is $2+o_d(1)$ \cite{Tao12}.
\Cref{fig:GOE} shows the shape $\tau$ associated with $G$, which simply consists of one edge.
We now proceed to bound \Cref{eq:Bq-restated}.

\parhead{Edge factor} According to our edge factor scheme described in \Cref{sec:edge-factor-scheme} (for $h_1$ edges), an $F/R/S$ step-label gets a factor of $\frac{1}{\sqrt{d}}$ while an $H$ step-label gets $\frac{2q}{\sqrt{d}}$.

\parhead{$\pur$ factor}
By \Cref{lem:pur-circle}, there is no $\pur$ factor for $F/R$, while $S$ and $H$ get $2$ and $1$ $\pur$ factors respectively.

\parhead{Vertex factor}
The weight of a circle vertex is $d$, thus any vertex making a first or last appearance gets a factor of $\sqrt{d}$.
We now case on the step-label and apply the vertex factor assignment scheme described in \Cref{sec:vertex-factor-scheme}.

\begin{itemize}
    \item $F$: the vertex in $U_\tau$ must be making a middle appearance; it is not first due to \Cref{def:vertex-appearance}, and it is not last as otherwise the edge appears only once throughout the walk. 
    The vertex in $V_\tau$ is making a first appearance, so it gets a factor of $\sqrt{d}$;
    
    \item $R$: the vertex in $V_\tau$ is making a middle appearance, since it is incident to an $R$ edge (hence not first appearance), and it is on the boundary hence bound to appear again the next block.
    The vertex in $U_\tau$ may be making its last appearance, so it gets a factor of $\sqrt{d}$;

    \item $S$: the vertex in $U_\tau$ is making a middle appearance (same as $F$), and the vertex in $V_\tau$ is making a middle appearance since it cannot be first and must appear again.
    In addition, it gets $2$ factors of $\pur$, which gives a bound of $(2q\cdot D_V)^{2}$;
    \item $H$: analogous to the above, both vertices are making middle appearance, and it gets $1$ factor of $\pur$, giving a bound of $2q\cdot D_V$.
    
\end{itemize}
Combining the vertex and edge factors, we can bound \Cref{eq:Bq-restated}:
\begin{align*}
    B_q(\tau) = 
    \sqrt{d} \cdot \frac{1}{\sqrt{d}} + \sqrt{d} \cdot \frac{1}{\sqrt{d}} + (2q\cdot D_V)^{2} \cdot \frac{1}{\sqrt{d}} + (2q\cdot D_V) \cdot \frac{2q}{\sqrt{d}}
    \leq 2 + o_d(1) \mcom
\end{align*}
since $q$ and $D_V$ are both $\polylog(d)$.
Therefore, by \Cref{prop:norm-bound-from-B}, we can conclude that $\|G\|_{\op} \leq 2 + o_d(1)$ with high probability, which is the correct bound.

\subsubsection{Bound for \texorpdfstring{$M_\beta$}{M\_beta}}

We now prove a bound on $\|M_{\beta}\|_{\op}$.
\Cref{fig:M_beta} shows the associated shape.

\begin{lemma}[Norm bound for $M_{\beta}$] \label{lem:M-beta-norm-bound}
    Let $m \geq \omega(d)$ and $q =  \Omega(\log^2 d)$.
    Then with probability $1-  2^{-q/\log d}$,
    \[ 
    \|M_\beta\|_{\op} \leq (1+o_d(1)) \frac{2m}{d^2}\,.
    \]
\end{lemma}
\begin{proof}
Let $\beta$ be the shape associated with the matrix $M_{\beta}$.
We bound $B_q(\beta)$ via our vertex/edge factor assignment schemes combined with $\pur$ factors.
Recall that each square vertex has weight $m$ and each circle vertex has weight $d$.
We case on the step-labels of the two edges,


\begin{enumerate}
    \item $F\rightarrow F$: we have an $F$ edge leading to a square vertex and circle vertex each.
    The first square vertex must be making a middle appearance (\Cref{def:vertex-appearance}), while the circle and the other square vertex make first appearances, giving a vertex factor of $\sqrt{m} \cdot \sqrt{d}$.
    Furthermore, there is no $\pur$ factor incurred.
    Finally, the edge factor is $\frac{2}{d^2}$.
    
    \item $R\rightarrow R$: we have both $R$ edges departing from a square and circle vertex.
    There is no vertex making first appearance and the square vertex on the right must be making middle appearance.
    The other two vertices may be making last appearances.    Furthermore, there is no $\pur$ factor incurred, while we assume each $R$ edge from a square vertex can be identified at a cost of $[2]$ modulo the ones with assigned $\pur$ factor, giving a total vertex factor of $2\sqrt{m} \cdot \sqrt{d}$.
   Finally, the edge factor is $\frac{2}{d^2}$.
    
    \item $R\rightarrow F$: 
    the circle vertex must be making a middle appearance, since the $F$ edge must be closed later.
    The square vertex on the left may be making a last appearance, and the square vertex on the right must be making a first appearance.
    This gives a vertex factor of $\sqrt{m}\cdot \sqrt{m} = m$.
    There is no $\pur$ factor, and the edge factor is $\frac{2}{d^2}$.


    \item $F\rightarrow R$: this cannot happen.
    The $F$ step means that the circle vertex is making its first appearance, but the $R$ step means that it must have appeared before.
    

    \item For any other step-labelings involving $S$ and $H$ step-labels, both vertices of the $S/H$ edge must be making middle appearances.
    Thus, the vertex factor is at most $\sqrt{m}$.
    By \Cref{lem:pur-circle,lem:pur-square}, the $\pur$ factor is at most $(2q )^4  $.
    Finally, the edge factor is at most $O(\frac{q^4}{d^2}) $.
    
\end{enumerate}
Summing over all possible step-labelings, by \Cref{prop:sum-of-labelings} we get
\begin{align*}
    B_q(\beta) \leq \sqrt{md} \cdot \frac{2}{d^2} + 2\sqrt{md} \cdot \frac{2}{d^2} + m \cdot \frac{2}{d^2} + \frac{q^{O(1)} }{d^2}
    \leq \frac{2m}{d^2} (1 + o_d(1)) \mcom
\end{align*}
provided that $m \gg d$.
Therefore, by \Cref{prop:norm-bound-from-B}, we have that with probability $1-  2^{-q/\log d}$, $\|M_{\beta}\|_{\op} \leq (1+o_d(1)) \frac{2m}{d^2}$.
\end{proof}

\subsubsection{Bound for \texorpdfstring{$M_\al$}{M\_alpha}}

We now prove a bound on $\|M_{\al}\|_{\op}$.
\Cref{fig:M_alpha} shows the associated shape.

\begin{lemma}[Norm bound for $M_\al$] \label{lem:M-alpha-norm-bound}
    Let $m \geq \omega(d)$ and $q = \Omega(\log^2 d) $.
    Then with probability $1-  2^{-q/\log d}$,
    \[ 
    \|M_\al\|_{\op} \leq (1+o_d(1)) \cdot \frac{1}{d^2} (3d\sqrt{m} + 2m) \,.
    \]
\end{lemma}
\begin{proof}
Let $\alpha$ be the shape associated with the matrix $M_\al$.
Similar to the proof of \Cref{lem:M-beta-norm-bound}, we bound $B_q(\al)$ by casing on the step-labelings.
There are two paths from $U_\al$ to $V_\al$, 4 edges in total.
t\begin{enumerate}
    \item In the case of all $F$ or all $R$, one of the square vertices must be making middle appearance, hence we get a vertex factor of $\sqrt{m} \cdot (\sqrt{d})^2 = d \sqrt{m}$. There is no $\pur$ factor, and the edge factor is $(\frac{1}{\sqrt{d}})^4 = \frac{1}{d^2}$.
    For the case of all $R$, by \Cref{cor:pur-square} we pick up an additional factor of $[2]$ since we assume each $R$ edge from a square vertex can be identified at a cost $[2]$ modulo those with assigned $\pur$ factor.

    \item If both paths are $R \to F$, then both circle vertices are making middle appearances, hence we get a vertex factor of $\sqrt{m} \cdot \sqrt{m} = m$.
    There is no $\pur$ factor while we pick up a factor $[2]$ for the return from the square vertex. Finally, the edge factor is $(\frac{1}{\sqrt{d}})^4 = \frac{1}{d^2}$.

    \item Analogous to $M_\beta$, an $F \to R$ path cannot happen.

    \item For any other step-labelings involving $S$ and $H$ step-labels, there must be at least one square and one circle vertex making middle appearances, so the vertex factor is at most $\sqrt{m} \cdot \sqrt{d}$.
    The $\pur$ factor is $(2q \cdot D_V)^{O(1)} = \polylog(d)$, and the edge factor is $(\frac{2q}{\sqrt{d}})^4 $.

\end{enumerate}
Summing over all possible step-labelings, by \Cref{prop:sum-of-labelings} we get
\begin{align*}
    B_q(\al) \leq d\sqrt{m} \cdot \frac{1}{d^2} + 2d\sqrt{m} \cdot \frac{1}{d^2} + 2m \cdot \frac{1}{d^2} + \sqrt{md} \cdot \frac{q^{O(1)} }{d^2}
    \leq (1+ o_d(1)) \cdot \frac{1}{d^2} (3d\sqrt{m} + 2m) \mper
\end{align*}
Therefore, by \Cref{prop:norm-bound-from-B}, we have that with probability $1-  2^{-q/\log d}$, $\|M_{\al}\|_{\op} \leq (1+o_d(1)) \cdot \frac{1}{d^2}(3d\sqrt{m} + 2m)$.
\end{proof}

\section{Matrix decomposition of \texorpdfstring{$\cR$}{R}}
\label{sec:R-decomposition}

In this section, we work towards analyzing $\cR$ and proving \Cref{lem:R-norm-bound}.
Recall our candidate matrix $\Lambda = I_d - \cR \in \R^{d \times d}$ from \Cref{def:candidate-matrix}, where
\begin{equation*}
    \cR \coloneqq \sum_{i=1}^m w_i v_iv_i^T = \sum_{i=1}^m \left( M^{-1} \eta \right)[i] \cdot v_iv_i^T \mper
\end{equation*}
See \Cref{eq:eta-vector,eq:M-matrix} for a reminder of the definitions of $M \in \R^{m\times m}$ and $\eta \in \R^m$.

To analyze $\cR$, we begin by obtaining an explicit expression for $M^{-1}$ using the Woodbury matrix identity (\Cref{fact:woodbury}), as discussed in \Cref{sec:M-inverse}.
Recall that we write $M = A + UCU^T$ where $A$ is positive definite with high probability by \Cref{lem:A-eigenvalues}, and $U = \frac{1}{\sqrt{d}} \begin{bmatrix} 1_m & \eta \end{bmatrix} \in \R^{m \times 2}$ and $C = \begin{bmatrix}
    1 & 1 \\ 1 & 0
\end{bmatrix}$.
Restating \Cref{eq:russ}, the scalars $r,s,u \in \R$ are defined as follows,
\begin{align*}
    \begin{bmatrix}
    	r & s \\
    	s & u
    \end{bmatrix}
    \coloneqq 
    C^{-1} + U^T A^{-1} U =
    \begin{bmatrix}
	\frac{1_m^T A^{-1} 1_m}{d} & 1 + \frac{\eta^T A^{-1} 1_m}{d}\\
	1+ \frac{\eta^T A^{-1} 1_m}{d} & -1 + \frac{\eta^T A^{-1} \eta }{d} 
    \end{bmatrix}
    \mper
    \numberthis \label{eq:russ-restated}
\end{align*}

Our next step is to show the following expansion of $M^{-1} \eta$.

\begin{proposition}[Expansion of $M^{-1}\eta$]  \label{prop:M-inverse-eta}
    Let $r,s,u\in \R$ be defined as in \Cref{eq:russ-restated}.
    Then,
    \begin{align*}
        M^{-1} \eta = \frac{r+s}{s^2-ru} \cdot A^{-1} \eta - \frac{u+s}{s^2-ru} \cdot A^{-1} 1_m \mper
        \numberthis \label{eq:M-inverse-expanded}
    \end{align*}
\end{proposition}
\begin{proof}
    The inverse of \Cref{eq:russ-restated} is as follows,
    \begin{align*}
        \Paren{C^{-1} + U^T A^{-1} U}^{-1} = \begin{bmatrix}
            r & s \\ s & u
        \end{bmatrix}^{-1}
        = \frac{1}{ru-s^2}
        \begin{bmatrix}
    		u & -s\\
    		-s & r
        \end{bmatrix} \mper
    \end{align*}
    Then, applying the Woodbury matrix identity (\Cref{fact:woodbury}), we have
    \begin{align*}
        M^{-1} &= A^{-1} - \frac{1}{ru-s^2} A^{-1} U \begin{bmatrix}
        	u & -s\\
        	-s & r
        \end{bmatrix} U^T A^{-1} \\
        &= A^{-1} + \frac{1}{s^2-ru} A^{-1} \left(  u \cdot \frac{1_m 1_m^T}{d} - s \cdot \frac{\eta 1_m^T + 1_m \eta^T}{d} + r \cdot \frac{\eta \eta^T}{d} \right) A^{-1} \mper
    \end{align*}
    Next, using the above, we have
    \begin{align*}
        M^{-1} \eta =
        A^{-1} \eta + \frac{1}{s^2-ru} \bigg( 
        & \Paren{ u\cdot \frac{1_m^T A^{-1} \eta}{d} - s \cdot \frac{\eta^T A^{-1} \eta}{d} } \cdot A^{-1} 1_m \\
        & + \Paren{ -s \cdot \frac{1_m^T A^{-1} \eta}{d} + r \cdot \frac{\eta^T A^{-1} \eta}{d} } \cdot A^{-1} \eta \bigg) \mper
    \end{align*}
    Plugging in the definition of $r,s,u$ in \Cref{eq:russ-restated}, we get
    \begin{align*}
        M^{-1}\eta &= A^{-1}\eta + \frac{1}{s^2-ru} (u (s-1)-s(u+1))  A^{-1}1_m + \frac{1}{s^2-ru}(-s (s-1) + r(u+1))A^{-1}\eta \\
	&= A^{-1}\eta - \frac{u+s}{s^2-ru} \cdot A^{-1}1_m + \frac{-s^2 + ru + r+s}{s^2-ru} \cdot A^{-1}\eta \\
	&= \frac{r+s}{s^2-ru} \cdot A^{-1} \eta - \frac{u+s}{s^2-ru} \cdot A^{-1} 1_m \mcom
    \end{align*}
    finishing the proof.
\end{proof}

\subsection{Inverse of \texorpdfstring{$A$}{A}: Neumann series and truncation}
\label{sec:A-inverse}

In light of \Cref{prop:M-inverse-eta}, we proceed to analyze $A^{-1}$.
Recall from \Cref{prop:M-decomposition} that $A = I_m + M_{\alpha} + M_{\beta} + M_D + \frac{1}{d}I_m$.
A useful tool to obtain inverses is to apply Neumann series (a.k.a.\ matrix Taylor expansion of $\frac{1}{1-x})$, which allows us to write \[ 
(I-T)^{-1} = \sum_{k=0}^\infty T^k
\]
for $\|T\|_{\op} < 1$.
In our case, let $T \coloneqq  -(M_{\alpha} + M_{\beta} + M_D + \frac{1}{d}I_m)$, then $\|T\|_{\op} < 1$ is guaranteed by \Cref{lem:T-norm}.
Thus, we can write
\[
    A^{-1} = \sum_{k=0}^\infty T^k = \sum_{k=0}^\infty (-1)^k \sum_{(Q_1,\dots,Q_k) \in \{M_{\alpha}, M_{\beta}, M_D, \frac{1}{d}I_m\}^k } Q_1 Q_2 \cdots Q_k \mper
\]
With the norm bounds from \Cref{lem:T-norm}, we can truncate the series, i.e., capping the number of occurrences of $M_{\alpha}, M_{\beta}, M_D$ and $\frac{1}{d}I_m$ by certain thresholds, such that the error is small.

\begin{definition}[Truncated $A^{-1}$]  \label{def:truncated-A-inverse}
    We define thresholds $\tau_1 = \tau_2 = O(\log d)$, $\tau_3 = 3$, and $\tau_4 = 1$,
    and define the truncation of $A^{-1}$ as
    \begin{equation*}
        T_0 \coloneqq \sum_{\substack{k_1,k_2,k_3,k_4\in \N \\ k_i \leq \tau_i,\ \forall i}} (-1)^{k_1+\cdots + k_4} \sum_{\substack{(Q_1,\dots,Q_k) \in \{M_{\alpha}, M_{\beta}, M_D, \frac{1}{d}I_m\}^k \\ \text{$i$-th matrix occurring $k_i$ times}}} Q_1 Q_2 \cdots Q_k \mper
        \numberthis \label{eq:truncated-A-inverse}
    \end{equation*}
\end{definition}

In the next lemma, we upper bound the truncation error.

\begin{lemma}[Truncation error of $A^{-1}$] \label{lem:truncated-A-inverse}
    Suppose $m \leq cd^2$ for a small enough constant $c$.
    Let $T_0$ be the truncated series defined in \Cref{def:truncated-A-inverse} with thresholds $\tau_1=\tau_2 = O(\log d)$, $\tau_3 = 3$ and $\tau_4 = 1$.
    Then, with probability $1-o_d(1)$, the truncation error $E = A^{-1} - T_0$ satisfies $\|E\|_{\op} \leq O(\frac{\log d}{d})^{2}$.
\end{lemma}
\begin{proof}
    From \Cref{lem:T-norm}, we know that $\|M_{\alpha}\|_{\op}, \|M_{\beta}\|_{\op} \leq 0.1$, $\|M_D\|_{\op} \leq O(\sqrt{\log d/d})$, and $\|\frac{1}{d}I_m\|_{\op} = \frac{1}{d}$ with high probability.
    We can bound the contribution of $\|M_{\alpha}\|_{\op}$ in the truncation error by
    \begin{align*}
        &\sum_{i=\tau_1+1}^\infty \sum_{j=0}^{\infty} \binom{i+j}{i} \cdot \|M_{\alpha}\|_{\op}^i \cdot \Paren{\|M_{\beta}\|_{\op} + \|M_{D}\|_{\op} +  \frac{1}{d} }^j  \\
        \leq\ &\sum_{i=\tau_1+1}^\infty \sum_{j=0}^{\infty} 2^{i+j} \cdot  \|M_{\alpha}\|_{\op}^i \cdot \Paren{0.1+ o(1)}^j \\
        \leq\ & \Paren{2\|M_{\alpha}\|_{\op}}^{\tau_1+1} \cdot O(1) \mper
    \end{align*}
    Similarly for $M_{\beta}, M_D$ and $\frac{1}{d}I_m$.
    Therefore, we can bound the total truncation error:
    \begin{align*}
        \|E\|_{\op} &\leq \Paren{ \Paren{2\|M_{\alpha}\|_{\op}}^{\tau_1+1} + \Paren{2\|M_{\beta}\|_{\op}}^{\tau_2+1} + \Paren{2\|M_{D}\|_{\op}}^{\tau_3+1} + (2/d)^{\tau_4+1} } \cdot O(1) \\
        &\leq O\Paren{\frac{\log d}{d}}^{2}
    \end{align*}
    by our choice of thresholds $\tau_1,\tau_2,\tau_3,\tau_4$.
\end{proof}

\subsection{Decomposition of \texorpdfstring{$\calR$}{R} via truncated \texorpdfstring{$A^{-1}$}{A\{-1\}}}

We shift our attention back to $\calR = \sum_{i=1}^m w_i v_i v_i^T$.
Using the expansion of $M^{-1}\eta$ in \Cref{prop:M-inverse-eta} (\Cref{eq:M-inverse-expanded}) and the truncation of $A^{-1}$, we decompose $\calR$ as follows.

\begin{proposition}[Decomposition of $\calR$] \label{prop:R-decomposition}
    Suppose $m \leq cd^2$ for a small enough constant $c$.
    Let $T_0$ be the truncated series of $A^{-1}$ defined in \Cref{def:truncated-A-inverse}. Then,
     \[
    \calR = \calR_1 + \calR_2 + E_{\calR} \mcom
    \]
    where
    \begin{align*}
        \calR_1 &\coloneqq \frac{r+s}{s^2-ru} \sum_{i\in [m]} (T_0 \eta)[i] \cdot v_i v_i^T \mcom \\
        \calR_2 &\coloneqq \frac{-u-s}{s^2-ru} \sum_{i\in [m]} (T_0 1_m)[i] \cdot v_i v_i^T \mcom
    \end{align*}
    and $\|E_{\calR}\|_{\op} \leq o(1)$ with probability $1-o_d(1)$.
\end{proposition}

We first state the following claim which is needed for the error analysis.
\begin{claim} \label{claim:sum-vi-vi}
    With probability $1-e^{-d}$, $\Norm{\sum_{i=1}^m v_iv_i^T}_{\op} \leq (1+o_d(1)) \frac{m}{d}$.
\end{claim}
\begin{proof}
    We can write $\sum_{i=1}^m v_i v_i^T = VV^T$ where $V \in R^{d\times m}$ has i.i.d.\ $\calN(0,\frac{1}{d})$ entries, and note that $\|\sum_{i=1}^m v_i v_i^T\|_{\op} = \sigma_{\max}(V)^2$.
    Since we assume $m \geq \omega(d)$, by standard concentration of the largest singular value of rectangular Gaussian matrices~\cite{Tro15}, we have that $\sigma_{\max}(V) \leq (1+o_d(1)) \sqrt{\frac{m}{d}}$ with probability at least $1-e^{-d}$.
\end{proof}
  

\begin{proof}[Proof of \Cref{prop:R-decomposition}]
    Recall that $\calR = \sum_{i\in[m]} (M^{-1}\eta)[i] \cdot v_i v_i^T$.
    Unpacking the expression of $M^{-1}\eta$ in \Cref{prop:M-inverse-eta} and plugging in $A^{-1} = T_0 + E$, we have $\calR = \calR_1 + \calR_2 + E_{\calR}$ where
    \begin{align*}
        E_{\calR} = \sum_{i\in[m]} \delta_i v_iv_i^T \mcom
        \quad 
        \delta_i \coloneqq \frac{r+s}{s^2-ru} \iprod{E[i],\eta} - \frac{u+s}{s^2-ru} \iprod{E[i], 1_m} \mper
    \end{align*}

    We first upper bound $|\delta_i|$ for all $i\in[m]$.
    First, observe that $|\iprod{E[i], \eta}| \leq \|E\|_{\op} \cdot \|\eta\|_2$ and $|\iprod{E[i], 1_m}| \leq \|E\|_{\op} \cdot \|1_m\|_2$.
    By \Cref{Claim:eta-bound,lem:truncated-A-inverse}, we have $\|E\|_{\op} \leq O(\frac{\log d}{d})^2$ and $\|\eta\|_2 \leq O(\sqrt{\frac{m}{d}})$ with high probability.
    Moreover, by \Cref{lem:M-invertible}, we have that $r = \Theta(\frac{m}{d})$, $|s| \leq O(\sqrt{d})$ and $-1 \leq u \leq -1/2$, hence $s^2 - ru \geq \Omega(\frac{m}{d})$.
    Thus,
    \begin{align*}
        |\delta_i| \leq O\Paren{\frac{\log d}{d}}^2 \cdot \Paren{\sqrt{\frac{m}{d}} + \frac{d^{3/2{}}}{\sqrt{m}}}
        \leq O\Paren{\frac{\log^2 d}{\sqrt{md}}}
    \end{align*}
    as we assume that $m \leq O(d^2)$.
    
    Finally, \Cref{claim:sum-vi-vi} states that $\|\sum_{i=1}^m v_iv_i^T\|_{\op} \leq O(\frac{m}{d})$.
    Since
    \begin{align*}
        -\sum_{i\in[m]} |\delta_i| v_i v_i^T  \preceq E_{\calR} \preceq \sum_{i\in[m]} |\delta_i| v_i v_i^T \mcom
    \end{align*}
    we may conclude that $\|E_{\calR}\|_{\op} \leq O(\frac{\log^2 d}{\sqrt{md}}) \cdot O(\frac{m}{d}) \leq O(\frac{\log^2 d}{\sqrt{d}})$ as $m \leq O(d^2)$.
\end{proof}

\label{sec:R-analysis}

\subsection{Overview: each term is a dangling path of injective gadgets}

By writing $A^{-1}$ as a truncated series (\Cref{def:truncated-A-inverse}) and the decomposition of $\calR$ (\Cref{prop:R-decomposition}), we may view $\calR_1$ and $\calR_2$ as linear combinations of $d \times d$ matrices of the forms
\begin{align*}
    \sum_{i\in[m]} \Paren{Q_1 Q_2 \cdots Q_k \eta}[i] \cdot v_i v_i^T
    \quad \text{and} \quad 
    \sum_{i\in[m]} \Paren{Q_1 Q_2 \cdots Q_k 1_m}[i] \cdot v_i v_i^T
    \numberthis \label{eq:R-terms}
\end{align*}
respectively, where $Q_1,\dots,Q_k \in \{M_\al, M_\beta, M_D, \frac{1}{d} I_m\}$ are $m \times m$ matrices.

Using the machinery of graph matrices described in \Cref{sec:graphical-matrices}, we can systematically represent these matrices as shapes (with underlying input $\{v_1,\dots,v_m\} \subseteq \R^d$).
For starters, the off-diagonal part of the matrix $\sum_{i=1}^m v_i v_i^T$ is represented by the shape in \Cref{fig:R-base}; it serves as the ``base'' for other matrices in \Cref{eq:R-terms}.
Each matrix in \Cref{eq:R-terms} can be represented by attaching ``gadgets'' (\Cref{fig:gadgets}) to the square vertex in the middle.
For the case of $\calR_1$, since $\eta_j = \|v_j\|_2^2-1 = \sum_{a=1}^d h_2(v_j[a])$ for $j\in [m]$ (see \Cref{eq:eta-vector}), each shape has an extra $h_2$ attached at the end.
\Cref{fig:R-A} shows an example of such a shape.

\begin{figure}[ht]
    \centering
    \begin{subfigure}[b]{0.28\textwidth}
        \includegraphics[width=\textwidth]{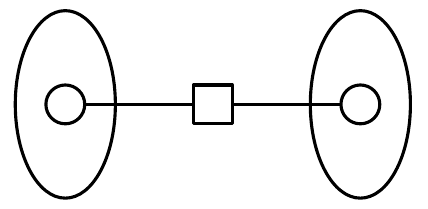}
        \caption{Off-diagonal part of $\sum_{i=1}^m v_i v_i^T$.}
        \label{fig:R-base}
    \end{subfigure}
    \quad
    \begin{subfigure}[b]{0.32\textwidth}
        \includegraphics[width=0.94\textwidth]{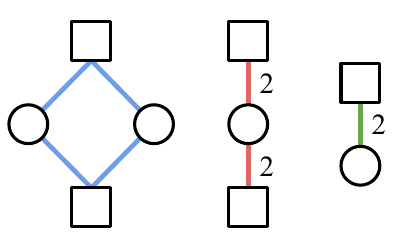}
        \caption{Left: $M_{\alpha}$ gadget. Middle: $M_{\beta}$ gadget. Right: final $h_2$ gadget.}
        \label{fig:gadgets}
    \end{subfigure}
    \quad
    \begin{subfigure}[b]{0.28\textwidth}
        \includegraphics[width=0.94\textwidth]{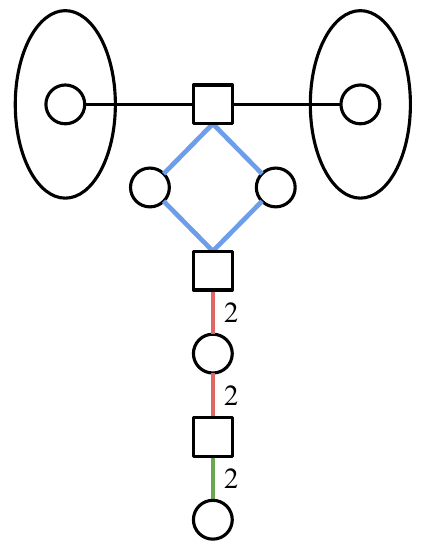}
        \caption{Off-diagonal part of $\sum_{i=1}^m (M_\al M_\beta \eta)[i] \cdot v_i v_i^T$.}
        \label{fig:R-A}
    \end{subfigure}
    \captionsetup{width=.9\linewidth}
    \caption{Examples of graph matrices in the decomposition of $\calR$.
    Recall that square vertices take labels in $[m]$ and circle vertices take labels in $[d]$.
    Unlabeled edges have Hermite index $1$.}
    \label{fig:examples}
\end{figure}

In the following, we formally define such shapes which we call \emph{dangling} shapes.

\begin{definition}[Injective gadgets]
We call each shape in $\{M_\al, M_\beta, M_D, \frac{1}{d}I_m\}$ an \emph{injective gadget} in the dangling shape.
\end{definition}

\begin{definition}[Dangling shapes]
    We further separate the matrices in the decomposition of $\calR$ (\Cref{eq:R-terms}) into diagonal and off-diagonal terms: 
    \begin{enumerate}
        \item For the off-diagonal terms, we start with a path from $U$ to $V$ (each containing a circle vertex receiving labels in $[d]$) passing through a middle square vertex that receives a label in $[m]$.
        \item For diagonal terms, we have a double edge (that shall be distinguished from an $h_2$ edge as the double edge here stands for $v_i[a]^2$) connected to a middle square vertex that receives a label in $[m]$.
    \end{enumerate}

    For each term, we specify a length-$k\geq 0$ dangling path that starts from the middle square vertex such that
    \begin{itemize}
        \item for any $k>0$, each step comes from one of the following gadgets in $\{ M_\al, M_\beta, M_D, \frac{1}{d}I_m\}$;
        \item The dangling path is not necessarily injective, that we may have each vertex appearing at multiple locations along the path.
        However, since it is a walk along the above gadgets, the path is \emph{locally injective} within each gadget.
    \end{itemize}
    For matrices from $\calR_1$, there is an additional $\eta$ factor.
    Thus, we attach an $h_2$ gadget (\Cref{fig:gadgets}) to the end of the dangling path.
    We call this the ``final $h_2$ gadget''.
\end{definition}

\begin{definition}[Gadget incursion] \label{def:gadget-incursion}
    Let $A,B$ be two gadgets along the dangling path.
    We call it a gadget-incursion if there are unnecessary vertex intersections between $V(A)$ and $V(B)$ beyond the ``necessary intersection'': when $A,B$ are adjacent, $V_A = U_B$ is the necessary boundary intersection, any intersection in $V(A)\setminus V_A $ and $V(B)\setminus U_B$ is a gadget incursion; similarly, when $A, B$ are not adjacent,  any intersection in $V(A) $ and $V(B)$ is a gadget incursion.
\end{definition}

\paragraph{Vertex and edge appearance}


Recall from \Cref{sec:overview-local-analysis} that we bound the norm of a graph matrix by analyzing length-$2q$ ``block-walks'' of the shape and bounding the vertex/edge factor of each ``block-step''.
To this end, we need to consider both \textit{global} and \textit{local} appearances of a labeled vertex.
We remind the reader that a labeled square (resp.\ circle) vertex is an element in $[m]$ (resp.\ $[d]$), and a labeled edge is an element in $[m] \times [d]$ (see \Cref{rem:labeled-vertices-edges}).

\begin{definition}[Local versus global appearance]
    Given a block-walk, we call each labeled vertex's appearance within the given block a local appearance, and each vertex's appearance throughout the walk a global appearance.

    Moreover, we say that a labeled vertex/edge is making its \emph{global} first/last appearance if it is the first/last appearance of that labeled vertex throughout the walk.
    Similarly, we say it is making its \emph{local} first/last appearance if it is the first/last appearance within the given block in the walk.
\end{definition}

We also need the following definition, which is a special case that we need to handle.
The term ``reverse-charging'' will be clear once we describe our edge charging scheme in \Cref{sec:analysis-of-R}.

\begin{definition}[Reverse-charging step/edge] \label{def:reverse-charging}
    For a given walk, and a block in the walk, we call a step $u\rightarrow v$ reverse-charging if 
   \begin{enumerate}
       \item the underlying edge is making its last \emph{global} appearance throughout the walk;
       \item the underlying edge's first \emph{global} appearance is also in the current block;
       \item (Reverse)  the first appearance of the underlying edge goes from $v$ to $u$.
   \end{enumerate}
\end{definition}


\subsubsection{Further set-up for step-labeling and edge factor scheme}
\label{sec:further-setup-edge-factors}

Our argument for tight matrix norm bounds requires assigning each edge (or step) a step-label (\Cref{def:step-labeling}) that represents whether it is making first/middle/last appearance, and assigning edge factors based on the edge type (recall our edge factor scheme in \Cref{sec:edge-factor-scheme}).
However, further care is warranted for dangling shapes when an edge appears with different Hermite indices in the walk (e.g., appears both as an $h_1$ and $h_2$ edge).
In this case, it is no longer true that an $h_k$ edge needs to appear at least twice in the walk for the random variable to be non-vanishing.
For example, suppose an edge appears as $h_1$, $h_1$, and $h_2$ in the walk.
Then, even though $h_2$ only appears once, this term is non-vanishing under expectation:
\begin{equation*}
    \E_{x\sim\calN(0,1/d)}\Brac{h_1(x)^2 h_2(x)} =
    \E\left[x^2 (x^2-\frac{1}{d}) \right] = \E\left[x^4\right] - \E\left[\frac{x^2}{d}\right] = \frac{2}{d^2}\,.
    \numberthis \label{eq:h1-h1-h2}
\end{equation*}


The matrices that arise in our analysis may contain $h_1, h_2, h_3, h_4$ edges.
For $i\leq 4$, we treat an $h_i$ edge as $i$ edge-copies in our edge factor assignment scheme.


\begin{definition}[Step-labeling scheme for mixed edges]  \label{def:step-labeling-mixed}
    For each step regardless of the Hermite index, assign a step-label to all its edge-copies as follows, 
    \begin{enumerate}
        \item Assign an $F$ step if it is making its first appearance;
        \item Assign an $H$ step if it is making its middle appearance
        \item Assign an $R$ step if it is making its last appearance.
    \end{enumerate}
\end{definition}

We next describe our edge factor assignment scheme.

\begin{lemma}[Edge factor assignment scheme]
\label{lem:mixed-edge-factor-scheme}
    For any graph matrix of size at most $D_V$ that contains $h_1,h_2,h_3,h_4$ edges and walks of length $2q \geq \Omega(\log d)$,
    we can assign values to each edge-copy among the edge's appearance throughout the walk such that
    \begin{enumerate}
        \item Each edge-copy of an $h_1, h_2$ edge with $F/R$ step-label gets assigned a value $\frac{2^{1/4}}{\sqrt{d}}$.
        \item Each edge-copy of an edge with step-label $H$ and any edge-copy of an $h_3, h_4$ edge gets assigned a value $\frac{32 q D_V}{\sqrt{d}}$.
        \item In the case of a random variable appearing as $h_1, h_1, h_2$ (in an arbitrary order), it gets assigned value of $\frac{2}{d^2}$ in total, in particular, each edge-copy of the $h_2$ edge gets assigned a value $\frac{\sqrt{2}}{\sqrt{d}}$.
    \end{enumerate}
\end{lemma}
\begin{proof}
    We first note that $\E_{x\sim\calN(0,1/d)}[h_k(x)^2] = \frac{k!}{d^k}$.
    For $k=1,2$, if an $h_k$ edge only appears twice and no other Hermite index occurs, then it must have $2k$ edge-copies with step-labels $F$ and $R$, giving an edge factor of $(\frac{2^{1/4}}{\sqrt{d}})^{2k} = \frac{2^{k/2}}{d^k}$, which is larger than $\frac{k!}{d^k}$ for $k = 1, 2$.
    
    Next, we consider the case when $h_3, h_4$ edges are involved or when an edge appears more than twice, i.e., some edge-copies are assigned step-label $H$.
    Let $a_k$ be the number of times $h_k$ appears in the walk, and let $t \coloneqq \sum_{k\leq 4} k \cdot a_k$ be the total number of edge-copies.
    Applying Cauchy-Schwarz twice and \Cref{fact:hermite-moments},
    \begin{align*}
        \Abs{ \E_{x\sim \calN(0,1/d)} \Brac{\prod_{k\leq 4} h_k(x)^{a_k}} }
        &\leq \E\Brac{h_1(x)^{2a_1} h_2(x)^{2a_2}}^{1/2} \E\Brac{h_3(x)^{2a_3} h_4(x)^{2a_4}}^{1/2} \\
        &\leq \prod_{k\leq 4} \E\Brac{h_k(x)^{4a_k}}^{1/4} 
        \leq \prod_{k\leq 4} \Paren{\frac{4k a_k}{d}}^{k\cdot a_k/2} \mper \\
        &\leq \Paren{\frac{4t}{d}}^{t/2} \mper
    \end{align*}
    We next show that the edge factors assigned to the $t$ edge-copies upper bound the above.
    Let $t_0$ be the number of edge-copies that get assigned $\frac{2^{1/4}}{\sqrt{d}}$.
    We must have $0 \leq t_0 < t$ and $t_0 \leq 4$.
    Then, the assignment scheme gives
    \begin{equation*}
        \Paren{\frac{2^{1/4}}{\sqrt{d}}}^{t_0} \cdot \Paren{\frac{32q D_V}{\sqrt{d}}}^{t-t_0}
        \geq d^{-t/2} (32q D_V)^{t-t_0} \mper
    \end{equation*}
    Since the length of the walk is $2q$ and the size of the graph matrix (shape) is $\leq D_V$, we have $t \leq 8q D_V$.
    Thus, if $t \leq 8$, then clearly $(32q D_V)^{t-t_0} \geq (4t)^{t/2}$;
    otherwise, $t-t_0 \geq t/2$ and $(32q D_V)^{t-t_0} \geq (4t)^{t/2}$.
    This shows that the edge factors correctly account for the values from the Hermite characters.

    For the special case when an edge appears as $h_1, h_1, h_2$, the factor $\frac{2}{d^2}$ follows from \Cref{eq:h1-h1-h2}.
    This completes the proof.
\end{proof}

\subsection{Local Analysis for \texorpdfstring{$\calR$}{R}} \label{sec:analysis-of-R}

\begin{lemma} \label{lem:block-val-bound-RA} For some constant $\eps>0$, for any $q<d^\eps$,  the block-value function for $R$ is bounded by 
\[	B_q(\calR ) \leq \frac{1}{10}\,. \]
\end{lemma}


We first state our vertex factor assignment scheme (with redistribution) which assigns vertex factors to labeled vertices according to their global appearances. It is the same one as described in \Cref{sec:vertex-factor-scheme}.

\begin{mdframed}[frametitle= {Vertex factor assignment scheme}]
\begin{enumerate}
    \item For each labeled vertex $i$ making its first or last \emph{global} appearance, assign a factor of $\sqrt{\wt(i)}$;
    \item For each labeled vertex $i$ making its middle \emph{global} appearance, assign a factor of $1$ if it is reached via an $R$ step, otherwise assign a factor $2q \cdot D_V$ as well as its corresponding $O(1)$ $\pur$ factors.
\end{enumerate}
\end{mdframed}


We next describe the scheme that assigns edge-copies to vertices.
In most cases, we charge edge-copies (on the dangling path) to the vertex it leads to, unless it is a \emph{reverse-charging} edge (\Cref{def:reverse-charging}).
Recall that we define a step $u\to v$ to be reverse-charging if it is making its last global appearance and if its first global appearance is also in the current block going from $v \to u$.

\begin{mdframed}[frametitle= {"Top-down" edge-copy charging primitive}]
\begin{enumerate}
    \item Assign both edges on the $U-V$ path to the first square vertex in the middle;
    \item For any step $u \to v$ before the final $h_2$ gadget, assign the edge to $v$ unless this is a reverse-charging edge (\Cref{def:reverse-charging}), in which case we assign to $u$;
    \item For the final $h_2$ gadget (if any), we reserve its assignment from the current scheme.  
\end{enumerate}
\end{mdframed}

%

See \Cref{sec:illustration} for an illustration of the above scheme.
We next show the following invariant throughout the walk, 

\begin{proposition}  \label{prop:top-down-charging}
We can assign each edge-copy to at most one vertex,
\begin{enumerate}
    \item for a circle vertex, it is assigned $1$ edge-copy if it is making first/last \emph{global} appearance, and $2$ edge-copies if both;
    \item for a square vertex on the dangling path, it is assigned $2$ edge-copies if it is making  first/last \emph{global} appearance, and $4$ edge-copies if both;
    \item for any square vertex's \text{global} middle appearance yet \text{local} first appearance in the block, it is assigned at least $1$ edge-copy;
    \item no surprise/high-mul step is assigned to any vertex's \emph{global} first/last appearance.
\end{enumerate}
\end{proposition}
\begin{proof}

We start by giving the argument when $U\neq V$, and then show it can be modified into an argument for $U=V$.

\parhead{Charging vertices outside $U\cup V$ and the final $h_2$ gadget} 
The above scheme applies immediately to vertices that appear in the current block while not in $U\cup V$ or the last square vertex as well as the final circle vertex it connects to in the $h_2$ gadget.	

It follows by observing that any circle (resp.\ square) vertex that makes its first \emph{global} appearance is the destination of $1$ (resp. $2$) edge-copies that are not reverse-charging, in which case the edge-copies are assigned to the particular vertex.
This holds analogously for vertices making the last \emph{global} appearance but not the first \emph{global} appearance in the block, since the edges are not reverse-charging.
We note that this is true for the first square vertex as well since we assign both edges on the $U-V$ path to the first square vertex.


 On the other hand, for vertices making both the first and the last \emph{global} appearances, consider the assignment until the final $h_2$ gadget, the charging above gives us $1$ and $2$ edges for each such vertex when it first appears \emph{locally}. Furthermore, notice that each of these edges need to be closed, and they are assigned to the destination vertex, i.e., the particular vertex in inspection, and this completes the proof of our assignment restricted to vertices outside $U\cup V$ and the final $h_2$ gadget.

\parhead{Charging the final gadget and finding "block-reserve"} \label{para:finding-block-reserve}
The goal is to identify an edge-copy as block-reserve, and assign vertex factors to corresponding edges in the final $h_2$ gadget (if any). We first consider the case when there is no final $h_2$ attachment, i.e.\ the term from $\frac{u+s}{s^2-ru} A^{-1} 1_m$.
Recalling our bounds from \Cref{lem:M-invertible} that $s^2-ru = \Omega\Paren{\frac{m}{d}}$, and $s=O(\sqrt{d}), |u|=O(1)$ (here the weaker bound on $s$ suffices), we observe that \[ 
\Abs{\frac{u+s}{s^2-ru}} = O\Paren{\frac{\sqrt{d}}{m/d}} = O\Paren{\frac{1}{\sqrt{d}}} \mcom
\]
and the normalizing constant can be regarded as an edge-copy. This is assigned as the block-reserve factor if there is no final $h_2$ gadget, and there is no vertex factor involved when there is no final $h_2$ gadget since this may be the last appearance of a square vertex while its factor has already been assigned to the edges that first lead to this vertex in the current block.

We now proceed to assign edges from the final $h_2$ gadget to vertex factors involved in the final $h_2$ gadget, moreover, we will identify one edge-copy of assigned factor at most $\frac{\sqrt{2}}{\sqrt{d}}$ as "block-reserve" that allows us to further charge vertices in $U\cup V$.
 
For the vertices in the final $h_2$ gadget, we observe the following, \begin{enumerate}
	\item in the final gadget, the last square vertex cannot be making its  first \emph{global} appearance yet it may be making its last appearance (since it is on the gadget boundary, we consider it appears in both the final $h_2$ attachment gadget and the gadget that precedes it);
	\item the last circle vertex may be making either its first or last \emph{global} appearance but it cannot be both;
	\item the $h_2$ edge in-between has not yet been assigned, and we can split it as two edge-copies, each of weight $\frac{2^{1/4}}{\sqrt{d}}$ if $F/R$ under the \emph{global} step-labeling.
	\end{enumerate}
We now observe the following,
\begin{enumerate}
    \item If the final $h_2$ attachment is not a reverse-charging edge assigned to the source square vertex, the square vertex has also been assigned by $2$ factors if making its first/last \emph{global} appearance, and $4$ if both, in the current block by the previous charging. In this case, note that we can reserve one of the two edge-copies, that is a factor of at most $\frac{2^{1/4}}{\sqrt{d}} $,\begin{itemize}
	    \item If the circle vertex is making first/last appearance, we either have both edge-copies receiving $F/R$ labeling, or a mix of $H,R$ labeling. In the first case, we have two assigned factors of $\frac{2^{1/4}}{\sqrt{d}}$. Assign one to the circle vertex's first/last appearance, and another to the block-reserve. In the second case, we have a mix of $H,R$ edges, that we at least have the underlying random variable appears twice as $h_1$ edges and at least once as an $h_2$ edge, by moving $\tilde{O}(1)$ polylog factors to the second appearance of the underlying random variable locally, we have again two edge copies of value $\frac{2^{1/4}}{\sqrt{d}}$, and the assignment follows from above.
	    \item If the circle vertex is making a middle appearance, we may assign related $\pur$ factors and vertex-factor cost to one edge-copy such that one edge-copy is of value $\Tilde{O}(\frac{1}{\sqrt{d}})$ while the other is still at most $\frac{2^{1/4}}{\sqrt{d}} $; assign the edge-copy with $\Tilde{O}(\frac{1}{\sqrt{d}})$ to the middle appearance factor and assign the edge-copy with factor $\frac{2^{1/4}}{\sqrt{d}} $ for block-reserve.
	\end{itemize}
	\item If the  final attachment $h_2$ edge is a reverse-charging step that corresponds to an edge whose $F$ copy 
	leads to the final square vertex, we assign both $h_2$ edges to the final square vertex as the square vertex may be making its last \emph{global} appearance. Note that if the square vertex makes its first \emph{global} appearance in the current block, it has already been assigned $2$ edges.
	
	 This leaves the last circle vertex potentially uncharged as it may be making its last appearance as well. Furthermore, we need to find one more "reserve" edge-copy for the block which would then be used to charge $U\cup V$.

    \item \textbf{Finding critical edge:} we now give a procedure to identify the critical edge, that is, an $h_2$ edge assigned to a circle vertex in the above top-down charging process. The process maintains a current circle vertex and its current gadget along the top-down path. In particular, the gadget considered is an $M_\beta$ gadget throughout. The process starts from vertex $s$, that is the circle vertex involved in the final $h_2$ attachment, and the gadget considered is the one in which $s$ opens up the $F$ step of the current edge $e^*$, the $h_2$ edge in the final gadget. Note that this is an $M_\beta$ gadget. We now case on the step-labeling of the top half of the $M_\beta$ gadget in inspection, in particular, the edges leading to the circle vertex $s$ in that particular gadget,
    \begin{itemize}
        \item This is an $F$, non-reverse-charging $R$, or $H$ step assigned to $s$, this is the critical edge we aim to find, as we have two edge-copies assigned to a circle vertex, and the process terminates.
        \item This is a reverse-charging $R$ edge: update the circle vertex to be the top vertex of the current $M_\beta$ gadget $s$, and update the edge $e^*$ to be the $h_2$ edge in the top half of the current gadget, and repeat the above process. Note that the updated gadget must appear, and additionally, on top of the current gadget in the dangling path.
    \end{itemize}
	We now observe that this process ultimately terminates since each time we move up towards the top of the dangling path. Once the critical edge is found, note that it contains two edge-copies, assign one to the vertex's factor, and another to the block-reserve. In the case of $H$ edges involved, locally assign the edge-value as $\frac{1}{\sqrt{d}}$ and $\tilde{O}(\frac{1}{\sqrt{d}})$, and assign the copy with $\frac{1}{\sqrt{d}}$ for block-reserve, while the other for vertex's factor.
	\end{enumerate}

\paragraph{Reserving surprise/high-mul step from vertex-factor assignment outside $U\cup V$} 
We first observe that in the above assignment, it is clear that the edges assigned to vertex's first appearance cannot be surprise-visit nor high-mul step. That said, it suffices for us to consider the assignment for vertex's last appearance factor (whose first appearance is not in the same block). Consider such vertex's first local appearance, they may either be $H/S$ edges if not $R$ under the global edge-labeling. If so, since the vertex is making the last appearance in the current block, each corresponds to an $R$-step in this block. Moreover, observe that any such $R$-step is intended for "reverse-charging" if the vertex in inspection is making both first and last global appearance in the block, and thus it is not assigned to the vertex factor of the source vertex. That said, it suffices for us to swap the $H/S$ step with the $R$ step so that no surprise/high-mul step is assigned to vertex's (polynomial) factor.

	At this point, for a block with $U\neq V$, we have assigned \begin{enumerate}
		\item $1$ or $2$ edge for each vertex's first/last appearance in the current block outside $U_\al\cup V_\al$ depending on the vertex's type;
		\item $1$ edge-copy of value $O(\frac{1}{\sqrt{d}})$ has been identified as the block-reserve.
		\item it is also straightforward to observe that any edge assigned for vertex's first/last appearance cannot be a surprise visit nor high-mul visit.
	\end{enumerate}

%

	\paragraph{Charging circle vertices in $U\cup V$}
	\begin{proposition}[Unbreakable $U-V$ path]
	There is always a $U-V$ path such that each edge along the path is of odd multiplicity. Call this path $P_{\textnormal{safe}}$.
\end{proposition}
\begin{proof}
 It suffices for us to restrict our attention to paths of only $h_1$ edges. By our gadget property, each vertex except $U_\al\cup V_\al$ is incident to an even number of $h_1$ edges, while only vertex copy in $U_\al\cup V_\al$ is of odd degree. Suppose the path is broken, consider the (maximal) component $C_U$ connected to $U_\al$ (but not $V_\al$), we first observe that the edge-multiplicity of $E(C_U, V(\al)\setminus C_U)$ must be even, as otherwise, some edge is of odd multiplicity, and the edge is included inside $C_U$ as opposed to be across the cut. That said, \[ 
 \sum_{v\in C_U} \deg(v) = 2 \cdot \mul(E(C_U)) + \mul(E(C_U, V(\al)\setminus C_U)) \mper
 \]
 Observe that the LHS is odd as any vertex copy except $U_\al$ has even degree, and we have exactly $1$ odd-degree vertex copy. On the other hand, RHS follows as each edge-multiplicity in $E(C_U)$ contributes a factor $2$ to degree of vertices in $C_U$ while edges across the cut contribute $1$ for each multiplicity; since $E(C_U, V(\al)\setminus C_U)$ is even by connectivity argument above, the RHS is also even and we have a contradiction.
\end{proof}
\begin{corollary}
	There is at least one vertex $v^*$ making middle-block appearance, i.e. it is a vertex that appears in previous blocks, and is appearing again in future blocks.
\end{corollary}

	We recall our charging so far that we have one edge-copy reserved, while we have two circle vertices in $U\cup V$ uncharged. And we now show that the single-edge copy is sufficient. i.e., at most one vertex factor is picked up among these two vertices.

	Consider the path $P_{\textnormal{safe}}$, and observe that it passes through both vertices in $U\cup V$. Take $v^*$ to be the first vertex on the path that pushes out an $F/H/S$ edge. In the case $v^* = U$, note that it suffices for us to assign the reserve factor for the vertex in $V$; similarly, assign the reserve factor for $U$ for $v^*= V$. In the case $v^*\notin U\cup V$, note that the previous scheme assigns at least one edge for $v^*$ if it is a circle, and $2$ if it is a square when $v^*$ first appears in the current block, and this is in fact not needed since $v^*$ is making a middle appearance. That said, we have at least $2$ factors, $1$ from the reserve factor, and $1$ from the edge-assignment for $v^*$ in the previous scheme, we now assign these two factors for $U \cup V$.
	
\begin{remark}
	$v^*$ is picked such that it is either reached by an $R$ edge, or it is the boundary vertex $U_\al$. In other words, no factor is needed for specifying $v^*$ of the block.
\end{remark}	
	
	and this completes the proof to our proposition when $U\neq V$.
	
	\paragraph{Analysis for $U=V$}
 We now consider the case when $U=V$.

\begin{definition}[$\mathsf{SQ}_1$: first square vertex in $R$ term]
    For each $R$ term, call the first square vertex on top of the dangling path the $\mathsf{SQ}_1$ vertex. 
\end{definition}
\begin{remark}
This vertex is an exception from any other square vertex as its first appearance might be the destination of two edge-copies that correspond to the same underlying random variable. This happens when an edge making its first and last appearance at the same time for the $U=V$ term.
\end{remark}

For $U=V$, the charging for vertices outside is identical except for the first square vertex $\mathsf{SQ}_1$: since the two edges assigned to it now correspond to the same underlying edge-copy, and may now receive $F, R$-labeling. That said, both edges are now assigned to the first square vertex. If the square vertex is making either first/last appearance but not both in the current block, the current assignment is sufficient. However, there are no reverse-charging edge copies protecting $\mathsf{SQ}_1$, and that warrants a further analysis.

We follow the previous strategy in identifying block-reserve: even though for $U=V$, the vertex $U=V$ is by definition not contributing any vertex-appearance factor,  we now intend the block-reserve factor to pay for either additional $\pur$ factor due to the dormant step stemming from $U=V$, or the last appearance factor of $\mathsf{SQ}_1$ but not both. 

\begin{enumerate}
    \item $\mathsf{SQ}_1$ makes both first and last appearance in the current block: in this case, the edges assigned to $\mathsf{SQ}_1$ in the beginning of the top-down charging process are assigned to the first appearance of  $\mathsf{SQ}_1$. That said, it remains for us to identify two extra edge-copies for the last appearance factor. 
    
\textbf{Charging for terms without final $h_2$ attachment:} these terms come with a normalization of $|\frac{u+s}{s^2-ru}| = O(\frac{1}{d})$, and these are two edge copies that we assign to the last appearance of $\mathsf{SQ}_1$. 

 \textbf{Charging for terms with final $h_2$ attachment:}
     we first note that any edge incident to $\mathsf{SQ}_1$ must be closed, and further case in whether there is any surprise visit arriving at $\mathsf{SQ}_1$ throughout the dangling path. 
\begin{enumerate}

%
        
		\item  Suppose that there is no surprise visit arriving at $\mathsf{SQ}_1$ throughout the dangling path.
		



		 Consider the final departure from  $\mathsf{SQ}_1$, and note that it pushes out at least two edge-copies.
  \begin{itemize}
      \item If this is an $M_\beta$ gadget, let $r$ be the circle vertex it connects to via an $h_2$ edge. Suppose the random variable corresponding to $(\mathsf{SQ}_1,r)$ first appears as an $h_2$ edge, in this case, the circle vertex has been assigned two edge-copies the first time it appears as it is reached using $h_2$ edge, that said, we can reserve the other two edge copies from the final attachment for the missing last appearance factor of $\mathsf{SQ}_1$;
          \item If this is an $M_\beta$ gadget, yet the random variable $(\mathsf{SQ}_1, r)$ first appears as an $h_1$ edge in the current block. Note that there is at least one more edge-copy of $h_1$ edge that is currently assigned an $H$-label, observe that we can replace its label to be $R$, and assign both $F/R$ copies to the vertex factor of the circle vertices. We now relabel the final $h_2$ step to be $H^*$, and note that these two edges get assigned a value at most $\frac{2}{d}$, and assign them both to the missing last appearance factor of $\mathsf{SQ}_1$. 
          \item To see that $H^*$ does not require extra $D_V\cdot q$ factor, we observe that this is a circle vertex traversed before in the block, and for each block, $H^*$ is unique. That said, it suffices for us to use a special label in $[2]$ when $H^*$ first appears in the block to identify the edge. Moreover, note that swapping $H^*$ with the $R$ step does not add to $\pur$ factor of the destination circle vertex as this edge no longer appears; 
              \item If this is an $M_\al$ gadget, let $r_1, r_2$ be the circle vertices the square vertex is
    connected to. Note that for each $r_i$, there must be at least two more edge-copies of random variable of $(\mathsf{SQ}_1, r_i)$ that receive $H$ label, i.e., we have a factor of $\tilde{O}(\frac{1}{d})$. That said, we may reassign the factors and extract one full factor of $\frac{1}{\sqrt{d}}$ (with the remaining being $\tilde{O}(\frac{1}{\sqrt{d}})$, and since we have two of them that combine to a factor of $\frac{1}{d}$, assign it to the missing last appearance factor of $\mathsf{SQ}_1$. 
  \end{itemize}

        \item There is a surprise visit arriving at $\mathsf{SQ}_1$. We first observe this must be either an $h_2$ edge, or a pair of distinct $h_1$ edges. The surprise visit must be closed already, and their underlying $R$ copies are intended for reverse-charging in the top-down charging scheme for the last appearance of $\mathsf{SQ}_1$. In this case, the factor for the last appearance has already been assigned edges. 
    \end{enumerate}
   \item  \textbf{Finding "block-reserve" for $U=V$:}
 $SQ_1$ is not making both first and last appearance in the current block, we first note that the departure from $U\cup V$ is special as this is the only vertex that can appear on the boundary again without being reached by any edge, and this is not a dormant gadget $M_D$. That said, the most recent departure may contribute $\pur$ factor to the vertex in $U\cup V$, and we now identify an edge-copy for this factor. 
\begin{enumerate}
 \item  \textbf{Charging for terms without final $h_2$ attachment:} for $A^{-1}1_m$ term where the final $h_2$ gadget is missing, we pick up a normalization constant of $O(\frac{1}{d})$ and we designate this as the block-reserve.
	
\item \textbf{Charging for terms with final $h_2$ attachment:} for terms with the final $h_2$ gadget, the analysis of finding block-reserve from the case $U\neq V$ applies identically for finding an $h_2$ edge that gets assigned to a circle vertex in the top-down traversal process. In particular, we may designate one edge-copy among the two copies of the identified $h_2$ edge as a block-reserve to charge the corresponding $\tilde{O}(1)$ factors from $\pur$.

\begin{remark}
	Note that the edge identified from the above process is in fact stronger, as it carries a factor of $O(\frac{1}{\sqrt{d}})$ as opposed to $\tilde{O}(\frac{1}{\sqrt{d}})$. That said, since for $U\neq V$ terms, the block-reserve is only assigned for $\tilde{O}(1)$ factors from $\pur$, either is sufficient.
\end{remark}
	
\end{enumerate}
\end{enumerate}	

\end{proof}

\subsection{Illustration via diagrams}
\label{sec:illustration}

\begin{center}
    \begin{tikzpicture}[
      mycircle/.style={
         circle,
         draw=black,
         fill=white,
         fill opacity = 1,
         text opacity=1,
         inner sep=0pt,
         minimum size=20pt,
         font=\small},
      mysquare/.style={
         rectangle,
         draw=black,
         fill=white,
         fill opacity = 1,
         text opacity=1,
         inner sep=0pt,
         minimum height=20pt, 
         minimum width=20pt,
         font=\small},
      myarrow/.style={-Stealth},
      node distance=0.6cm and 1.2cm
      ]
      \draw (-5,2.5) ellipse (.4cm and .6cm);
      \draw (0,2.5) ellipse (.4cm and .6cm);
      \draw[orange] (-2.5,-6) .. controls (-6.5,-4) and (-6.5,-2) .. (-2.5,0);
      \draw[orange] (-4, -1.5) -- (-4, -4.5);
      \draw[orange] (-1, -1.5) -- (-1, -4.5);
      
      \draw (5,2.5) ellipse (.4cm and .6cm);
      \draw[orange] (2.5,-6) .. controls (6.5,-4) and (6.5,-2) .. (2.5,0);
      \draw[orange] (4, -1.5) -- (4, -4.5);
      \draw[orange] (1, -1.5) -- (1, -4.5);
      
      \draw[orange] (-5,2.5) .. controls (-1,4) and (1,4) .. (5,2.5);
      \draw[orange] (-2.5, 0) -- (2.5, 0);
      \draw[orange] (-2.5, -8) -- (2.5, -8);  
      
      \node[mycircle]  at (-5, 2.5) (u) {$u$};
      \node[mycircle]  at (0, 2.5) (v) {$v$};
      \node[mysquare]  at (-2.5, 0) (a) {$a$};
      \node[mycircle]  at (-4, -1.5) (b) {$b$};
      \node[mycircle]  at (-1, -1.5) (c) {$c$};
      \node[mysquare]  at (-2.5, -3) (d) {$d$};
      \node[mycircle]  at (-4, -4.5) (e) {$b$};
      \node[mycircle]  at (-1, -4.5) (f) {$c$};
      \node[mysquare]  at (-2.5, -6) (g) {$a$};
      \node[mycircle]  at (-2.5, -8) (t) {$t$};
      
      \node[mycircle]  at (5, 2.5) (w) {$u$};
      \node[mysquare]  at (2.5, 0) (a2) {$a$};
      \node[mycircle]  at (1, -1.5) (b2) {$b'$};
      \node[mycircle]  at (4, -1.5) (c2) {$c'$};
      \node[mysquare]  at (2.5, -3) (d2) {$d'$};
      \node[mycircle]  at (1, -4.5) (e2) {$b'$};
      \node[mycircle]  at (4, -4.5) (f2) {$c'$};
      \node[mysquare]  at (2.5, -6) (g2) {$a$};
      \node[mycircle]  at (2.5, -8) (t2) {$t$};
    \foreach \i/\j/\txt/\p/\pp in {
      u.315/a.135/ /above/0.5,
      v.225/a.45/ /above/0.5,
      a.225/b.45/ /above/0.5,
      a.315/c.135/ /above/0.5,
      b.315/d.135/ /above/0.5,
      c.225/d.45/ /above/0.5
    }
    \draw[cyan] [myarrow] (\i) -- node[font=\small,\p,pos=\pp] {\txt} (\j);
    \foreach \i/\j/\txt/\p/\pp in {
      e.45/d.225/ /above/0.5,
      f.135/d.315/ /above/0.5,
      g.135/e.315/ /above/0.5,
      g.45/f.225/ /above/0.5
    }
    \draw[red] [myarrow] (\i) -- node[font=\small,\p,pos=\pp] {\txt} (\j);
    \foreach \i/\j/\txt/\p/\pp in {
      g.270/t.90/2/right/0.5
    }
    \draw[green] [myarrow] (\i) -- node[font=\small,\p,pos=\pp] {\txt} (\j);
    
    \foreach \i/\j/\txt/\p/\pp in {
      v.315/a2.135/ /above/0.5,
      w.225/a2.45/ /above/0.5
    }
    \draw[blue] [myarrow] (\i) -- node[font=\small,\p,pos=\pp] {\txt} (\j);
    \foreach \i/\j/\txt/\p/\pp in {
      a2.225/b2.45/ /above/0.5,
      a2.315/c2.135/ /above/0.5,
      b2.315/d2.135/ /above/0.5,
      c2.225/d2.45/ /above/0.5
    }
    \draw[cyan] [myarrow] (\i) -- node[font=\small,\p,pos=\pp] {\txt} (\j);
    \foreach \i/\j/\txt/\p/\pp in {
      e2.45/d2.225/ /above/0.5,
      f2.135/d2.315/ /above/0.5,
      g2.135/e2.315/ /above/0.5,
      g2.45/f2.225/ /above/0.5
    }
    \draw[red] [myarrow] (\i) -- node[font=\small,\p,pos=\pp] {\txt} (\j);
    \foreach \i/\j/\txt/\p/\pp in {
      g2.270/t2.90/2/right/0.5
    }
    \draw[green] [myarrow] (\i) -- node[font=\small,\p,pos=\pp] {\txt} (\j);
    \end{tikzpicture}
    \end{center}

    This figure illustrates the charging scheme. The light blue edges are $F$ edges and we assign their factor to their destination. The dark blue edges are $R$ edges whose first appearance is in a different block. We also assign these edges to their destination. The red edges are reverse-charging $R$ edges whose first appearance is in the current block. We assign these edges to the destination of the corresponding $F$ edge (which is generally the source for this edge).

    Note that two edge factors are missing from $v$. We obtain these factors from the two green edges pointing towards $t$.
    
    \begin{center}
    \begin{tikzpicture}[
      mycircle/.style={
         circle,
         draw=black,
         fill=white,
         fill opacity = 1,
         text opacity=1,
         inner sep=0pt,
         minimum size=20pt,
         font=\small},
      mysquare/.style={
         rectangle,
         draw=black,
         fill=white,
         fill opacity = 1,
         text opacity=1,
         inner sep=0pt,
         minimum height=20pt, 
         minimum width=20pt,
         font=\small},
      myarrow/.style={-Stealth},
      node distance=0.6cm and 1.2cm
      ]
      \draw (0,2) ellipse (.4cm and .6cm);
      \draw (0,2) ellipse (.5cm and .75cm);
      \draw[orange] (0,-6) .. controls (-4,-4) and (-4,-2) .. (0,0);
      \draw[orange] (-1.5, -1.5) -- (-1.5, -4.5);
      \draw[orange] (1.5, -1.5) -- (1.5, -4.5);
      \draw[orange] (0, -8) -- (1.5, -4.5);
      \node[mycircle]  at (0, 2) (u) {$u$};
      \node[mysquare]  at (0, 0) (a) {$a$};
      \node[mycircle]  at (-1.5, -1.5) (b) {$b$};
      \node[mycircle]  at (1.5, -1.5) (c) {$c$};
      \node[mysquare]  at (0, -3) (d) {$d$};
      \node[mycircle]  at (-1.5, -4.5) (e) {$b$};
      \node[mycircle]  at (1.5, -4.5) (f) {$c$};
      \node[mysquare]  at (0, -6) (g) {$a$};
      \node[mycircle]  at (0, -8) (t) {$c$};
    \foreach \i/\j/\txt/\p/\pp in {
      a.225/b.45/ /above/0.5,
      a.315/c.135/ /above/0.5,
      b.315/d.135/ /above/0.5,
      c.225/d.45/ /above/0.5
    }
    \draw[cyan] [myarrow] (\i) -- node[font=\small,\p,pos=\pp] {\txt} (\j);
    \foreach \i/\j/\txt/\p/\pp in {
      e.45/d.225/ /above/0.5,
      f.135/d.315/ /above/0.5,
      g.135/e.315/ /above/0.5
    }
    \draw[red] [myarrow] (\i) -- node[font=\small,\p,pos=\pp] {\txt} (\j);
    \foreach \i/\j/\txt/\p/\pp in {
      g.45/f.225/ /above/0.5,
      t.90/g.270/2/left/0.5
    }
    \draw[brown] [myarrow] (\i) -- node[font=\small,\p,pos=\pp] {\txt} (\j);
    
      \draw (4,2) ellipse (.4cm and .6cm);
      \draw (4,2) ellipse (.5cm and .75cm);
      \draw[orange] (4, -2) -- (6, 0);
      \node[mycircle]  at (4, 2) (u2) {$u'$};
      \node[mysquare]  at (4, 0) (a2) {$a'$};
      \node[mycircle]  at (6, 0) (b2) {$b'$};
      \node[mysquare]  at (4, -2) (t2) {$b'$};

    \foreach \i/\j/\txt/\p/\pp in {
      a2.00/b2.180/2/above/0.5
    }
    \draw[cyan] [myarrow] (\i) -- node[font=\small,\p,pos=\pp] {\txt} (\j);
    \foreach \i/\j/\txt/\p/\pp in {
      t2.90/a2.270/2/right/0.5
    }
    \draw[brown] [myarrow] (\i) -- node[font=\small,\p,pos=\pp] {\txt} (\j);      
    \end{tikzpicture}
    \end{center}
    
    \begin{center}
    \begin{tikzpicture}[
      mycircle/.style={
         circle,
         draw=black,
         fill=white,
         fill opacity = 1,
         text opacity=1,
         inner sep=0pt,
         minimum size=20pt,
         font=\small},
      mysquare/.style={
         rectangle,
         draw=black,
         fill=white,
         fill opacity = 1,
         text opacity=1,
         inner sep=0pt,
         minimum height=20pt, 
         minimum width=20pt,
         font=\small},
      myarrow/.style={-Stealth},
      node distance=0.6cm and 1.2cm
      ]
      \draw (-1.5,1.5) ellipse (.4cm and .6cm);
      \draw (1.5,1.5) ellipse (.4cm and .6cm);
      \draw[orange] (0,-4.5) .. controls (-2,-3.5) and (-2,-2.5) .. (0,-1.5);
      \node[mycircle]  at (-1.5, 1.5) (u) {$u$};
      \node[mycircle]  at (1.5, 1.5) (v) {$v$};
      \node[mysquare]  at (0, 0) (a) {$a$};
      \node[mycircle]  at (0, -1.5) (b) {$b$};
      \node[mysquare]  at (0, -3) (c) {$c$};
      \node[mycircle]  at (0, -4.5) (t) {$b$};
    \foreach \i/\j/\txt/\p/\pp in {
      u.315/a.135/ /above/0.5,
      v.225/a.45/ /above/0.5,
      b.270/c.90/2/right/0.5
    }
    \draw[cyan] [myarrow] (\i) -- node[font=\small,\p,pos=\pp] {\txt} (\j);
    \foreach \i/\j/\txt/\p/\pp in {
      t.90/c.270/2/right/0.5
    }
    \draw[red] [myarrow] (\i) -- node[font=\small,\p,pos=\pp] {\txt} (\j);
    \foreach \i/\j/\txt/\p/\pp in {
      a.270/b.90/2/right/0.5
    }
    \draw[green] [myarrow] (\i) -- node[font=\small,\p,pos=\pp] {\txt} (\j);
    \end{tikzpicture}
    \end{center}

\subsection{\texorpdfstring{$\pur$}{Pur} bound for square vertices in \texorpdfstring{$\calR$}{R}}
The prior $\pur$ bound does not apply well for square vertices in $\calR$, in particular, it should be pointed out that even before non-trivial intersection within each block along the dangling path, the $\pur$ argument based on $1$-in-$1$-out (or its slightly generalized version of $2$-in-$2$-out) falls apart in the analysis for $\calR$.

In particular, one may consider a walk on $R_S$ with the square vertex along the $U-V$ path fixed throughout the walk. It is easy to verify the the fixed square vertex may have growing unclosed $F$ edges without any surprise visit/high-mul step. That said, it is not sufficient to use the slack from such factors to offset the potential confusion due to $R$ edges. 

In this section, we give a new argument to handle the $\pur$ factor in the vanilla setting of $R$ when there is no non-trivial intersection along the dangling gadget-path, and then extend it for the general cases of $R$ where each block is not necessarily injective due to intersection across gadgets. 

\paragraph{Gap from square middle appearance} For starters, we observe that a vertex may only push out unforced returns if it is making a middle appearance given that it maintains a list of incident edges, including the additional information which are closed. When it is making the last appearance, any currently unclosed edge needs to be closed, and therefore shall be pushed out, giving us a fixed set of edges being pushed-out (where we momentarily ignore the question whether one needs to distinguish among the edges in the edge-set). This immediately renders us the following bound on $\pur$ readily,
\begin{lemma}[$\pur$ factor via middle appearance: beginner version]
For any vertex $v$, let $\mathsf{MidApp}(v)$ be the number of middle appearances of $v$ throughout the walk, we have \[ 
\pur(v) \leq 3 \cdot \mathsf{MidApp}(v)
\] 
provided each square vertex pushes out at most $3$ edges in each block throughout the walk, and each block is vertex-injective.
\end{lemma}

With the above $\pur$ bound, it may not be meaningful if we cannot obtain a slack from $\mathsf{MidApp}$. Fortunately, this is indeed the case for our setting, and we first observe this in the vanilla setting where there is no gadget-incursion within each block, (in particular, this already applies immediately if the edges are injective witin each block), \begin{enumerate}
    \item We have assigned each square vertex at least $1$ edge when it is making a middle appearance: to see this, the top-down charging scheme assigns each square vertex $2$ edges when the square-vertex appears for the first time in the block; in the case of charging $U\cup V$, $1$ edge may be re-routed from a square vertex making a \emph{global} middle appearance. That said, at least $1$ edge is assigned to each square vertex when it makes a middle appearance.
    \item Note that for a fixed vertex $v$ in a given block, its \emph{local} first appearance at the given block may be corresponding to a \emph{global} middle appearance, and such \emph{mismatch} is the source of our slack;
    \item Observe that each middle-appearance corresponds to a mismatch described above (provided each block is injective), and each such mismatch of local-global first appearance assigns a vertex making global middle appearance one edge-copy, that is a factor of $O(\frac{1}{\sqrt{d}})$;
    \item However, since each vertex's middle appearance does not get assigned any vertex factor in our scheme, we may use the $\frac{1}{\sqrt{d}}$ gap to offset the $3$ $\pur$ factors corresponding to the particular \emph{global} middle appearance, that is \[ 
    O(\frac{1}{\sqrt{d}}) \cdot (q\cdot D_V)^3  = o_d(1)\,.
    \]
\end{enumerate}
\paragraph{Extension to gadget intersections} To handle dangling paths with potentially intersecting gadgets, it should be pointed out the $\pur$ bound goes through as stated while we do not necessarily have a gap from middle-appearance. In particular, it is possible now that in a given block, a square vertex appears for various times and only has $2$ edges assigned to it for its first \text{local} appearance, as any of its subsequent appearance in the given block follow via closing some $F$ edge that gets opened up earlier in the work, and thus assigned to the destination as opposed to the given square vertex.

Towards generalizing the prior argument, we consider a specific subclass of \emph{global} middle appearance of a vertex through the block walk,
\begin{definition}[Middle appearance for a square vertex]
    For a given walk and a given block-step $\mathsf{BlockStep}_i$, we say a labeled vertex $v\in [m]$ makes a middle appearance in $\mathsf{BlockStep}_i$ if \begin{enumerate}
        \item $v$ makes appearance at $\mathsf{BlockStep}_i$;
        \item $v$ makes appearance at some $\mathsf{BlockStep}_j$ for $j<i$, and at some  $\mathsf{BlockStep}_{j'}$ for $j'>i$.
    \end{enumerate}
\end{definition}
    With some abuse of notation, we continue to let $\mathsf{MidApp}(v)$ denote the number of middle appearances of $v$. Note that this is clear when we are working with blocks that do not have block-injectivity, while it is equivalent to the previous definition in vertex-injective blocks.
    
In particular, we emphasize that following the above definition, in the case of a labeled vertex first appears at $\mathsf{BlockStep}_i$, and appears multiple times in various gadgets in the dangling-path of $\mathsf{BlockStep}_i$, it is not considered as making a middle appearance at $\mathsf{BlockStep}_i$.

\begin{lemma}[$\pur$ bound for square vertices]
For any square vertex $v$ that does not push out dormant gadgets, at any time-$t$, 
\[  
\pur_t(v) \leq  P_t(v) \leq  3 \cdot \mathsf{MidApp}_t(v) + 3(s_t(v) + h_t(v)) 
\]
where we define \[ 
P_t(v)  \coloneqq \#(\text{$R$ steps closed from $v$ by time $t$}) + \#((\text{unclosed edges incident to $v$ at $t$ }) )  - 4 \mper
\]
In other words, by assuming at $4$ possible return legs to be fixed each time a vertex is on the boundary, the number of unforced returns from a vertex by time $t$ is at most $3 \cdot \mathsf{MidApp}_t(v) + 3(s_t(v) + h_t(v))$.
\end{lemma}

\begin{proof}
    The first inequality is definitional as we assume each vertex may have $4$ edges being fixed, which incurs a cost of $[4]$ for each vertex each time it pushes out an $R$ step. Analogous to previous $\pur$ bounds, the base case is immediate when vertex first appears in the walk. In particular, we note that the above bound can be strengthened for the $\mathsf{BlockStep}_i$ in which vertex $v$ makes (globally) its first appearance. 
    
\begin{claim}
	For any square vertex $v$, for the $\mathsf{BlockStep}_i$, at any time-$t$ within the $\mathsf{BlockStep}_i$ , 
\[  
P_t(v) \leq  3(s_t(v) + h_t(v))  - 2 \mper
\]
\end{claim}
\begin{proof}
Notice this is immediate when vertex $v$ first appears, as it is incident to $2$ edges, and thus we have $P^{(1)}_t = -2 $ (as we have a $-2$ term since we maintain $3$ edges to be fixed instead of just $2$). The invariant holds as any subsequent departure opens up at most $2$ $F$ edges, and any subsequent arrival either closes both edges, or either arrival is along a surprise/high-mul visit, and give a net-gain of at most $3$ in the number of unclosed $F$ edges. This proves our claim.
\end{proof} 

It remains for us to consider the appearance of $v$ in subsequent blocks, in particular, we start with the \emph{locally} first appearance of the subsequent block. Let $t$ be the time-mark in which $v$ is making its first \emph{local} appearance in the current-block. Notice at time $t$ when the vertex first appears, it may be arrived via a single edge (as opposed to $2$ due to the $U-V$ path), and therefore it may push out at most $3$ edges.

    
	\paragraph{Appearance at the second block}
	If not, this is currently the second block in which $v$ appears: applying the claim on the first block, and observe that the most recent departure opens up at most $2$ $F$ edges, while the current arrival closes $1$ (unless $H$ or $S$, in which case a net-gain of $3$ suffices), we have \[ 
	P_t(v) \leq 3(s_t(v)+h_t(v)) -2 + 2 -1 =3(s_t(v)+h_t(v)) - 1
	\]
	where the $+2$ corresponds to the $2$ $F$ edges opened up at the most recent departure, and the $-2$ term corresponds to the term in the hypothesis on first-block, and the $-1$ comes from the current arrival closing at least one $R$ edge.
	
	For any subsequent appearance of $v$ in the current block, if any, the following is immediate,  \[ 
P_{t'}(v) - P_t(v) \leq 3 (s_{t'}(v) - s_t(v)) + 3 (h_{t'}(v) - h_t(v)) + 1
\]
as we observe that \begin{enumerate}
    \item The departure from the first vertex may open up at most $3$ edges instead of $2$;
    \item Any subsequent departure and arrival closes $2$ edges, and opens up at most $2$ edges, hence the previous argument applies, giving a net-gain of $+1$ due to the extra opening in the first departure.
\end{enumerate}
That said, for any appearance of $v$ in the second block at time $t'$, we have \[ 
P_{t'}(v)  =  P_{t(v)} + (P_{t'}(v) - P_t(v)) \leq 3(s_{t'}(v)+h_{t'}(v)) 
\]
	
\paragraph{Appearance at the future blocks} For any block, let $t_0$ be the local first appearance, and $t_1$ be the local final appearance, applying the above argument gives \[ 
P_{t_1}(v) - P_{t_0}(v) \leq  3 (s_{t_1}(v) - s_{t_0}(v)) + 3 (h_{t_1}(v) - h_{t_0}(v)) + 1\,.
\]
That said, it suffices for us to bound $P_{t_0}(v)$. This is bounded by \[ 
P_t(v) \leq  3 \cdot \mathsf{MidApp}_t(v) + 3(s_t(v) + h_t(v)) 
\]
Consider the base case when $v$ appears at the third block, the most recent departure opens up at most $2$ new $F$ edges. To offset this, we use the gain in $ \mathsf{MidApp}_t(v) $, as the appearance of $v$ in the second block is now counted as a middle appearance once $v$ appears in the third block, which gives a $+3$ on the RHS. The bound extends to any subsequent block immediately.
	This completes our proof of $\pur$ bound.
%
\paragraph{Extension to dormant gadgets}
To capture the $\pur$ change due to dormant gadget, we note that for each square vertex at each block, we can assume one dormant edge it pushes out being fixed while assign a $\pur$ factor for any other dormant gadget it pushes out at that block. Notice this is a cost at most $2$ for each square vertex throughout the walk, as we may need to assume $1$ dormant edge for its first appearance, and another for its last appearance. For any middle appearance, it suffices for us to assign a $\pur$ factor for any dormant edge it pushes out, as opposed to all-but-one in the case of first/last appearance. This prompts to define the following counter,
\begin{definition}[Dormant-excess]
    For each square vertex $v$, at any block $\mathsf{BlockStep}_i$, let $D_i$ be the number of dormant-gadgets it pushes out at this block,
    define $D_i(v) \coloneqq D_i - 1[ \text{first/last appearance at i} ]$, and additionally, the counter function is defined for the whole walk by taking \[
    D(v)  = \sum_{i:v \text{ appears in  } \mathsf{BlockStep}_i} D_i(v) \mper
    \]
\end{definition}
\begin{remark}
	By using an extra additive constant of $2$, each return using dormant leg is either forced or accounted for in $D(v)$.
\end{remark}
\begin{claim}
	Each dormant-excess gadget corresponds to a circle vertex reached using an $h_2$ edge, and gets assigned a factor of at most $\tilde{O}(\frac{1}{\sqrt{d}})$ in the combinatorial charging argument.
\end{claim}
\begin{proof}
Note that in the combinatorial charging argument, we have assigned both $h_2$ edges to the circle vertex's vertex factor, unless the square vertex is potentially making its first and last appearance at the same block, which is not ruled out by the definition of dormant excess. That said, for the circle vertex, it gets assigned at most a factor of $\sqrt{d}$ in our scheme, while both edges assigned at least $\tilde{O}(\frac{1}{\sqrt{d}})$ each, and combining the above yields the desired.
\end{proof}
\begin{corollary}
	By assuming at most $6$ return legs to be forced,  the number of unforced-return from $v$ throughout the walk is at most\[ \pur(v) \leq  3\cdot \mathsf{MidApp}_t(v) + 3(s_t(v) + h_t(v))  + D(v)\,. \]
\end{corollary}
This allows us to effectively ignore $\pur$ factor in the block-value analysis, as each $\pur$ factor can be distributed among  surprise visit, high-mul visit, or middle appearance such that each is assigned at most $3$ $\pur$ factors. Moreover, since each comes with a $O(\frac{1}{\sqrt{d}})$ factor, combining with the assigned $\pur$ factor contributes an $o_d(1)$ term provided $\frac{q^3}{\sqrt{d}} =o_d(1)$, which is sufficient for us as we set $q = d^{\eps}$ for a small enough constant $\eps$.
\end{proof}

\subsection{Wrapping up}
Given the edge-assignment scheme to vertex appearance, we now show why this immediately gives the desired bound on $B(\calR)$: we have the following factors, \begin{enumerate}
	    \item Each circle vertex gives a factor of $\sqrt{d}$ for their first/last appearance; the assignment gives each vertex one \emph{global} $F/R$ edge-copy each, that is a factor of  $\frac{\sqrt{2}}{\sqrt{d}}$, giving a bound of \[
	 \sqrt{d} \cdot \frac{\sqrt{2}}{\sqrt{d}} = \sqrt{2}\,;
	    \]
	    \item Each square vertex gives a factor of $\sqrt{m}$ for their first/last appearance, and the assignment above gives  \emph{global} $F/R$ each two edge-copies each, that is a factor of \[ 
	 \sqrt{m} \cdot    \frac{2}{d} \leq \frac{2\sqrt{ m}}{d}\,;
	    \]
	    \item Each square vertex that makes a \emph{global} middle appearance gives a $O(1)$ factors of $\pur$, while each such vertex appearance is assigned one edge, that is a factor of \[
	   (2q\cdot D_V)^{O(1)} \cdot  O\left( \frac{\sqrt{2} q^2}{\sqrt{d} }\right) \ll 1\,.
	    \]
	\end{enumerate}
	
\begin{corollary}
For each vertex's appearance and edges assigned to it, we have a factor of at most \[ 
(1+o_d(1)) \cdot   \sqrt{d} \cdot \frac{\sqrt{2}}{\sqrt{d}} \cdot 12 \leq (1+o_d(1)) \cdot 6\sqrt{2}
\]
for a circle vertex, and \[ 
(1+o_d(1)) \cdot 6 \cdot   \sqrt{m} \cdot \frac{2}{d } \leq (1+o_d(1)) \frac{12\sqrt{m}}{d}
\]
for a square vertex where the factor $6$ for each vertex comes from our bound that each vertex arrived using a forced $R$ edge can be specified at a cost of $[6]$.
\end{corollary}
\begin{proof}[Proof to \cref{lem:block-val-bound-RA}]
By our edge-copy charging scheme and \Cref{prop:top-down-charging}, we observe the following,
\begin{enumerate}
    \item Treat the $U-V$ path as a gadget, and we have $2$ choices depending on whether $U=V$ for the upcoming block;
    \item For each gadget-step, we sum over the edge-labelings;
    \item For each gadget along the dangling path, we pick up at most the following factor, \begin{itemize}
        \item $M_\al$: it gives $2$ circle vertices and $1$ square vertex, which is a factor of at most \[ 
      B_\al \coloneqq   (1+o_d(1)) \cdot 3^{4}\cdot  \frac{4}{d^2} \cdot \sqrt{md^2} \cdot 6^3 \,;
        \]
        \item $M_\beta$: it gives $1$ circle vertex and $1$ square vertex, which is a factor of at most \[ 
     B_\beta \coloneqq    (1+o_d(1)) \cdot 3^4 \cdot \frac{2}{d^2} \cdot \sqrt{md} \cdot 6^2  \,;
        \]
     	\item $M_D$: it gives $1$ circle vertex, which is a factor of at most $ 
     	(1+o_d(1)) \cdot 3^2    	$
     	if it is the first $M_D$ gadget of the block with the factor $3$ counting the step-label of the edge-copies in the gadget, or a factor of at most $B_D \coloneqq \tilde{O}(\frac{1}{\sqrt{d}})$ for any subsequent $M_D$ gadgets. 
    \end{itemize}
    \item For the $U- V$ path, the square vertex $SQ_1$ gives a factor of at most \[ 
    (1+o_d(1))  \cdot 6 \cdot \sqrt{\frac{m}{d^2}} 
    \] 
    and the circle vertex combined gives a factor of at most \[ 
    (1+o_d(1)) \cdot  (\sqrt{2})^2 \cdot  6
    \]
    where the factor of $6$ comes from at most one circle vertex being reached using $R$ edge.
\end{enumerate}
Therefore, combining the above bounds, we have \begin{align*}
	B(\calR) \leq (1+o_d(1)) \cdot 6 (\sqrt{\frac{m}{d^2}} \cdot 6 \cdot \sqrt{2}^2) \cdot 3^2 \cdot \prod_{\leq \polylog d \text{ gadgets} } (B_\al +B_\beta +B_D  )  < \frac{1}{2}
\end{align*}
provided \begin{align*}
	B_\al, B_\beta < \frac{1}{2} 
\end{align*}
and 
\[
 6 (\sqrt{\frac{m}{d^2}} \cdot 6 \cdot \sqrt{2}^2) \cdot 3^2  < \frac{1}{2}
 \]
 It can be verified that it suffices for us to take $m < \frac{1}{7000000}\cdot  d^2$ though we do not emphasize upon the particular constant as we believe a more careful argument by tracking our above bounds can render an improved constant without much work.
\end{proof}

\begin{lemma}[Restatement of ~\cref{lem:R-norm-bound}]
    With probability at least $1 - 2^{-d^\eps}$ for some constant $\eps>0$, \[ 
    \|\calR\|_{\op} <\frac{1}{2}\,.
    \]
\end{lemma}
\begin{proof}
    This follows by setting $q= d^{\eps}$ for some constant $\eps>0$ in \cref{prop:norm-bound-from-B}, and combines with our block-value function for small enough constant $c$.
\end{proof}

\section*{Acknowledgements} We would like to thank anonymous reviewers for their various suggestions in polishing our writing. J.H., P.K., and J.X. are grateful to Prayaag Venkat for bringing this problem to their attention in his theory lunch talk at CMU.

\bibliographystyle{alpha}
\bibliography{main}

\newpage
\appendix
\section{Deferred calculations}
\label{sec:missing-proofs}

\begin{claim}[Restatement of \Cref{Claim:eta-bound}]
 With probability at least $1-o_d(1)$, \[ 
 \|\eta \|_2^2 \leq (1+o_d(1)) \frac{2m}{d} \mper
 \]
 \end{claim}
 \begin{proof}
 We first unpack the inner product, \[ 
 \|\eta \|_2^2 = \sum_{i\in [m]} \eta_i^2 = \sum_{i\in [m]}\left( \sum_{j\in [d]} (v_i[j]^2-\frac{1}{d})\right)^2 = \sum_{i\in [m]} \sum_{j_1, j_2 \in [d]}\left(v_i[j_1]^2-\frac{1}{d} \right)\left(v_i[j_2]^2 -\frac{1}{d}   \right)\,.
 \]
  For any $q>0$, we then have \[ 
  \E[\|\eta\|_2^{2q}]  = \E\left[\sum_{s_1,s_2,\dots, s_q \in [m]} \sum_{\substack{a_1,a_2, \dots, a_q \in [d] \\ b_1, b_2,\dots,b_q \in [d] }} \prod_{i=1}^{q} (v_{s_i}[a_i]^2 - \frac{1}{d} ) \cdot (v_{s_i}[b_i]^2 - \frac{1}{d} )  \right] \leq\left( 1+o_d(1)\right)\left(\frac{2md}{d^2}\right)^q
  \]
  where the final bound follows in a similar way to our norm bounds as 
  \begin{enumerate}
      \item For each new $s_i$, we case on whether $a_i = b_i$, and if so, we pick a factor of $d$ for $a_i=b_i$ and a factor of $m$ for $s_i$;
      \item Otherwise, if we have $a_i\neq b_i$, since each edge is mean $0$, we can assign a factor of $\sqrt{d \cdot 2q}$ to each of the $a_i, b_i$ for their first two appearances, and a factor of $q$ for $a_i$,  $b_i$ making their third appearances and onwards; similarly, we assign a factor of $\sqrt{m \cdot q} $ for the first two appearances of $s_i$ or $q$ if $s_i$ is not making the first two  appearances;
      \item Each $H_2$ edge $(v_{s_i}[t_i]^2 - \frac{1}{d} )$ gets assigned its standard deviation that is $\frac{\sqrt{2}}{d}$ when it appears for the first two times, otherwise a factor of $O(\frac{q}{d})$;
      \item To see the final bound, notice we can case on whether each $H_2$ edge is making their first two appearances. \begin{itemize}
          \item In the case both edges making first two appearances, we pick up an edge-value of exactly $(\frac{\sqrt{2}}{d})^2= \frac{2}{d^2}$; otherwise, we pick up an edge-value at most $(\frac{q}{d})^2= \frac{q^2}{d^2}  $;
          \item Assuming both edges making first two appearances, and if $a_i=b_i$ matches, we pick up a factor of \[ 
          (1+o_d(1)) \frac{2}{d^2} \cdot md
          \]
          \item Assuming both edges making first two appearances while $a_i\neq b_i$, we pick up a factor of at most \[ 
          (1+o_d(1)) \frac{2}{d^2} \cdot \sqrt{md^2 q^3}
          \]
          \item If some edge is making third appearance or beyond, we pick up a factor at most  \[ 
          (1+o_d(1)) \frac{q^2}{d^2} \cdot O(q^2\sqrt{dq})
          \]
          as we have at most one new ''circle'' vertex that takes label in $[d]$ and gets assigned weight at most $\sqrt{dq}$.
          \item Summing over the above cases gives us a bound of \[ 
          (1+o_d(1)) \frac{2}{d^2} \cdot md = (1+o_d(1)) \frac{2m}{d}
          \]
      \end{itemize}
  \end{enumerate}

  Taking the $\frac{1}{q}$-th root and apply Markov's gives us the desired.
 \end{proof}

\begin{claim}[Bound of $s$]  \label{claim: bound-for-s}
	Recall that $s = 1+ \frac{1}{d}\eta^TA^{-1}1_m$, we have \[
	|s| \leq 1+o_d(1)
	 \]
\end{claim}
\begin{proof}
	It suffices for us to bound $ |\frac{1}{d}\eta^TA^{-1}1_m| \leq o_d(1)$. We adopt the truncation strategy in our main analysis of $\calR$, and split it as \[  
	\frac{1}{d}\eta^TA^{-1}1_m  = \frac{1}{d}\eta (T_0 + E) 1_m = \frac{1}{d}\eta T_0 1_m + \frac{1}{d}\eta E 1_m
	\]
	where we focus on the analysis of the first term.
\end{proof}

\begin{lemma}
\[ B\Paren{\frac{1}{d} \eta T_0 1_m} \leq \tilde{O}\Paren{\frac{1}{\sqrt{d}}} \]
\end{lemma}
\begin{proof}
    We mimic our analysis for the well-conditionedness of $A^{-1}$, in particular, each term $\eta T_0 1_m$ is a floating component and moreover, we apply our edge-assignment scheme from analysis of $M$ by traversing from the circle-vertex on one end (for concreteness, take it to be the one from $\eta^T$). With some abuse of notation, let this vertex be $\mathsf{circ}_1$ again. Consider the following assignment-scheme, \begin{enumerate}
        \item Assign each $F, H$ and non-reverse-charging $R$ to the vertex it leads to;
        \item Assign each reverse-charging $R$ step to the destination vertex of the underlying $F$ copy;
        \item Assign the normalizing constant of $\frac{1}{d}$ to the starting circle vertex $\mathsf{circ}_1$.
    \end{enumerate}
    The desired block-value bound then follows immediately by the following proposition.
    \begin{proposition}
    Each vertex making first/last appearance outside $\mathsf{circ}_1$ is assigned the required (non-surprise) $F/R$ edge-copies. Moreover, there is either an extra edge-copy of factor $\tilde{O}(\frac{1}{\sqrt{d}})$, or $\mathsf{circ}_1$ is making first/last but not both appearances while it is assigned a factor of $\frac{1}{d}$.
    \end{proposition}
    \begin{proof}
        The first part follows immediately from our analysis of $M$, that we assign each vertex's first factor to the edge that discovers it, and if the vertex is making the last appearance in the same block, the factor is assigned by the $R$ step of the edge that discovers it. To see that the assignment of last appearance is using $R$ edge alone when the first, last appearances are not in the same block, note that each vertex is assigned the first edge ''locally'' that discovers the first vertex. In the case that this edge is non-$R$, the edge must be closed later on and we can swap the assignment of the underlying $R$ copy of that edge. The swapping factor is valid as the $R$ copy of the edge is intended for protecting the vertex's both first and last appearance in the current block, which is again not needed in this case.
        
        To see the second part of the claim, note that the traversal path of the floating component starts from a circle vertex while ends at a square vertex. On the one hand, note that it is immediate if the first circle vertex $\mathsf{circ}_1$ is making first/last but not both appearances in the current block as it is assigned the normalizing constant of $\frac{1}{d}$, giving a gap of $\frac{1}{\sqrt{d}}$. On the other hand, if the $\mathsf{circ}_1$ is making both first and last appearances,  there are two cases, \begin{enumerate}
            \item The final departure from  $\mathsf{circ}_1$  is a reverse-charging step, in this case, there is a surprise visit within this block, and it gives an extra copy of edge-value $\tilde{O}(\frac{1}{\sqrt{d}})$;
            \item The final departure from $\mathsf{circ}_1$ is not reverse-charging, this edge either makes a middle appearance before as an $H$-step, or its underlying $F$ copy is a surprise visit arriving at $\mathsf{circ}_1$. In either case, this gives an extra edge-copy of factor $\tilde{O}(\frac{1}{\sqrt{d}})$.
        \end{enumerate}
    \end{proof}
\end{proof}

\section{Analysis of \texorpdfstring{$M_D$}{M\_D}}
\label{sec:M-D}

Recall from \Cref{prop:M-decomposition} that $M_D[i,i] = \|v_i\|_2^4 - \frac{2}{d} \|v_i\|_2^2 - 1$.
For ease of technical analysis, it is helpful to note here that we can further decompose $M_D$ to be the following matrices:
\begin{proposition}[Decomposition of $M_D$] \label{prop:D-decomposition}
    We write 
    \begin{align*}
        M_D = M_{D,1} + M_{D,2} + \Paren{2 + \frac{2}{d}} M_{D,3} \mcom
    \end{align*}
    where for any $i \in [m]$,
    \begin{enumerate}
        \item $M_{D,1} [i,i ] =\sum_{a\neq b\in [d]} \left(v_i[a]^2 -\frac{1}{d}\right) \left(v_i[b]^2 - \frac{1}{d}\right)$;
        \item $M_{D,2} [i,i ] =\sum_{a\in [d]} (v_i[a]^4 - \frac{6}{d}v_i[a]^2 + \frac{3}{d^2})$;
        \item $M_{D,3} [i,i ] =\sum_{a\in [d]} (v_i[a]^2 - \frac{1}{d})$.
    \end{enumerate}
\end{proposition}
\begin{proof}
    $M_D[i,i] = \|v_i\|_2^4 - \frac{2}{d} \|v_i\|_2^2 - 1$.
    We first unpack $\|v_i\|_2^4$: $\|v_i\|_2^4 = (\sum_{a\in[d]} v_i[a]^2)^2 = \sum_a v_i[a]^4 + \sum_{a\neq b} v_i[a]^2 v_i[b]^2$.
    On the one hand, for any $a\neq b \in [d]$,
    \begin{align*}
        v_i[a]^2v_i[b]^2 &= \left((v_i[a]^2-\frac{1}{d})+\frac{1}{d}\right)
        \cdot \left((v_i[b]^2-\frac{1}{d})+\frac{1}{d}\right) \mcom
    \end{align*}
    thus,
    \begin{align*}
    \sum_{a\neq b \in [d]} v_i[a]^2v_i[b]^2 = \underbrace{\sum_{a\neq b\in [d]}\left((v_i[a]^2-\frac{1}{d})(v_i[b]^2-\frac{1}{d})  \right)}_{M_{D,1}}  + \underbrace{2\cdot \frac{d-1}{d}\sum_{a}(v_i[a]^2-\frac{1}{d}) }_{(2-\frac{2}{d}) M_{D,3}}+ \frac{d(d-1)}{d^2} \mper
    \end{align*}    
    On the other hand,
    \begin{align*}
        \sum_{a\in [d]}v_i[a]^4 &= \sum_{a\in [d]}\left( (v_i[a]^4 - \frac{6v_i[a]^2}{d} + \frac{3}{d^2}) + \frac{6}{d}(v_i[a]^2 - \frac{1}{d}) + \frac{3}{d^2} \right)\\
        &= \underbrace{\sum_{a\in [d]} (v_i[a]^4 - \frac{6v_i[a]^2}{d} + \frac{3}{d^2})}_{M_{D,2} } + \underbrace{\frac{6}{d} \sum_{a\in [d]} (v_i[a]^2 - \frac{1}{d})}_{\frac{6}{d} M_{D,3} } + \frac{3}{d} \mper
    \end{align*}
    Thus, $\|v_i\|_2^4 = M_{D,1} + M_{D,2} + (2+\frac{4}{d})M_{D,3} + 1 + \frac{2}{d}$.
    
    Next, observe that $\|v_i\|_2^2 = M_{D,3}[i,i] + 1$.
    Thus, we have
    \begin{align*}
         M_D[i,i] = \|v_i\|_2^4 - \frac{2}{d}\|v_i\|_2^2 - 1
         = M_{D,1} + M_{D,2} + (2+\frac{2}{d})M_{D,3} \mcom
    \end{align*}
    completing the proof.    
\end{proof}

\begin{lemma}[Norm bound for $M_{D}$]  \label{lem:block-bound-M-D}
    For matrices $M_{D,1}, M_{D,2}, M_{D,3}$ defined in \Cref{prop:D-decomposition},
    we have the following bounds
    \begin{enumerate}
        \item $B(M_{D,1}) \leq \widetilde{O}(\frac{1}{d^2} \cdot \sqrt{d^2}) = \widetilde{O}(\frac{1}{d})$;
        \item $B(M_{D,2}) \leq \widetilde{O}(\frac{\sqrt{d}}{\sqrt{d}^8} )$;
	\item $B(M_{D,3}) \leq \widetilde{O}(\frac{\sqrt{d}}{d}) = \widetilde{O}(\frac{1}{\sqrt{d}})$.
     \end{enumerate}
     As a consequence, $\|M_D\|_{\op} \leq \widetilde{O}(\frac{1}{\sqrt{d}}) =  o_d(1)$ with probability $1-  2^{-q/\log d}$.
\end{lemma}
\begin{proof}
	We observe the following,\begin{enumerate}
		\item The square vertex is in $U\cap V$ and hence bound to be making a middle appearance (ignoring the first and the last block of the walk as they are offset by long-path);
		\item Each circle vertex is either making a first/middle/last appearance, and note that instead of keeping track of $\pur$ factor, it suffices for us to assign it a cost of at most $\sqrt{n} \cdot (q\cdot D_V)$;
		\item For $M_{D,1}$, each edge gets assigned a factor of at most  $\frac{q}{d}$, and we have at most two circle vertices outside making first/middle/last appearances, which give a bound of $O(\frac{q^2}{d^2} \cdot \sqrt{d^2} (q\cdot D_V)^2)$.
		\item $M_{D,3}$ is identical to $M_{D,1}$ except there is only one circle vertex outside, which gives a bound of $O(\frac{q}{d}\cdot \sqrt{d}\cdot q\cdot D_V  ) $.
		\item For $M_{D,2}$, each edge gets assigned a factor of at most  $\frac{q}{\sqrt{d}^8}$, and we have one single circle vertex outside, which combines to a bound of $O(\frac{q}{\sqrt{d}^8 } \cdot \sqrt{d} \cdot (q\cdot D_V))$.
	\end{enumerate}
	Finally, recalling that we set $q^{O(1)} \cdot D_V \ll  \sqrt{d} $ completes the proof to our lemma.
\end{proof}


\section{Sketch of an Alternative Analysis}
\begin{definition}
For each chain, we define local $F$, $R$, $S$, and $H$ edges as follows. Here we traverse the chain from top to bottom.
\begin{enumerate}
\item We say that an edge is a local $F$ edge if it is appearing for the first time in the current chain and its destination is appearing for the first time in the current chain.
\item We say that an edge is a local $R$ edge if it is appearing for the last time in the current chain.
\item We say that an edge is a local $H$ edge if it appears again both before and afterwards in the current chain.
\item We say that an edge is a local $S$ edge if it is appearing for the first time but its destination has appeared before in the current chain.
\end{enumerate}
For our local analysis, there are two circle vertices which we may need to be careful about.
\begin{definition}
Let $v_{left}$ be the circle vertex which is at the top left, let $v_{right}$ be the circle vertex which is at the top right (which may be equal to $v_{left}$), and let $v_{bottom}$ be the circle vertex which is at the bottom.
\end{definition}
We can always take the minimum weight vertex separator to be $v_{left}$, so we do not need to worry about $v_{left}$. However, as we will see, we have to be careful about $v_{right}$ and $v_{bottom}$.
\begin{definition}
We say that an edge is locally vanishing if we have that after the local intersections, the product of it and the edges parallel to it has a nonzero constant term.
\end{definition}
\begin{example}
Some examples of locally vanishing edges are as follows.
\begin{enumerate}
\item Two parallel edges with label $1$ are locally vanishing as $x^2 = \sqrt{2}\frac{x^2 - 1}{\sqrt{2}} + 1$
\item Two parallel edges with label $2$ are locally vanishing as 
\[
\left(\frac{x^2 - 1}{\sqrt{2}}\right)^2 = \sqrt{6}\left(\frac{x^4 - 6x^2 + 3}{\sqrt{24}}\right) + 2\sqrt{2}\left(\frac{x^2 - 1}{\sqrt{2}}\right) + 1
\]
\item More generally, two parallel edges are locally vanishing if they have the same label $k$ and are not locally vanishing if they have different labels.
\item Three parallel edges with labels $1$, $1$, and $2$ are locally vanishing as 
\[
x^2\left(\frac{x^2 - 1}{\sqrt{2}}\right) = 2\sqrt{3}\frac{x^4 - 6x^2 + 3}{\sqrt{24}} + 5\frac{x^2 - 1}{\sqrt{2}}+\sqrt{2}
\]
\end{enumerate}
\end{example}
\begin{definition}
We say that a vertex $v$ is locally isolated if $v$ is not equal to the top left circle or the top right circle and all edges incident to $v$ are locally vanishing.
\end{definition}
\end{definition}
For all edges except for two special edges $e^{*}_r$ and $e^{*}_b$ which we describe below, we assign them as follows.
\begin{enumerate}
\item For each local $F$ edge, we assign it to its destination. For the edge at the top between $v_{right}$ and its neighboring square vertex, we consider its destination to be the square vertex.
\item For each local $R$ edge, we assign it to its origin unless it is a dangling edge with label $2$. If it is a dangling edge with label $2$, we split it between its endpoints.
\item For each local $S$ and $H$ edge, we assign half of it as a bonus to its origin and half of it as a bonus to its destination.
\end{enumerate}
\begin{lemma}
All vertices except $v_{right}$ and $v_{bottom}$ have the required edge factors. 
\end{lemma}
\begin{proof}
For each circle vertex $u$ except for $v_{left}$ (which we doesn't need any edges), $v_{right}$, and $v_{bottom}$, consider the first and last time $u$ appears in the current chain. The first time $u$ appears, there must be a local $F$ edge pointing to it which gives $u$ one edge. If $u$ only appears once, it cannot be locally isolated, so it only needs one edge. Otherwise, the last time $u$ appears, it can only be locally isolated if all edges incident to it are $R$ edges. In this case, $u$ obtains a second edge.

Similarly, for each square vertex $v$, consider the first and last time $v$ appears in the current chain. The first time $v$ appears, there must be two local $F$ edges with label $1$ or one local $F$ edge with label $2$ pointing to it. The last time $v$ appears, it can only be locally isolated if all edges incident to it are $R$ edges, in which case $v$ obtains two additional edges.
\end{proof}
\begin{remark}
Note that $v_{right}$ is an exception because the first time it appears, it does not have a local $F$ edge pointing to it. Similarly, $v_{bottom}$ is an exception because the last time it appears, it may be incident to an $R$ edge with label $2$ without gaining an additional edge.
\end{remark}
\begin{definition}
When $v_{left} \neq v_{right}$ and $v_{right}$ is not equal to the final circle vertex, we find/define $e^{*}_r$ through the following iterative process. We start with the edge $e$ between $v_{right}$ and its neighbor. Note that this is a local $F$ edge going down from a vertex $v = v_{right}$ which is not locally isolated and which is not the final circle vertex. We now have the following cases
\begin{enumerate}
\item $e$ does not appear again and the other endpoint of $e$ is the final circle vertex. In this case, we take $e^{*}_{r} = e$ and assign one edge factor from it to $v_{right}$. If it has label $2$, we keep one edge factor in reserve in case $e^{*}_{b} = e^{*}_{r} = e$.
\item $e$ does not appear again and its other endpoint is not the final circle vertex. In this case, letting $v'$ be the other endpoint of $e$, we take $e'$ to be an edge going down from $v'$. If $e'$ is not a local $F$ edge, we take $e^{*}_{r} = e'$ and assign it to $v_{right}$. We can do this because $v'$ is not locally isolated so it already has all of the edge factors it needs. If $e'$ is a local $F$ edge, we again have a local $F$ edge $e'$ going down from a vertex which is not locally isolated and is not the final circle vertex, so we repeat this process.
\item $e$ appears again. In this case, let $e'$ be the next time $e$ appears. If $e'$ is not a local $R$ edge whose destination is $v$, we take $e^{*}_{r} = e'$ and assign it to $v_{right}$. If $e'$ is a local $R$ edge whose destination is  $v$, we let $e''$ be an edge going down from $v$.

If $e''$ is not a local $F$ edge then we take $e^{*}_{r} = e''$ and assign it to $v_{right}$. If $e''$ is a local $F$ edge then we are still in the situation where we have a local $F$ edge going down from a vertex which is not locally isolated and is not equal to the final circle vertex, so we repeat this process.
\end{enumerate}
\end{definition}
\begin{definition}
We find/define $e^{*}_b$ through the following iterative procedure. We start with the edge $e$ with label $2$ between $v_{bottom}$ and its neighbor. Note that $e$ must be a local $R$ edge. We then do the following.
\begin{enumerate}
\item If $e$ is not a locally vanishing edge then we take $e^{*}_b = e$. We assign half of $e^{*}_b = e$ to $v_{bottom}$ and keep half in reserve in case $e^{*}_b = e^{*}_r$.
\item If $e$ is a locally vanishing edge then $e$ must appear above. Consider the previous time $e$ appears and let $e'$ be this copy of $e$. There are a few possibilities for $e'$.
\begin{enumerate}
\item $e'$ is an edge with label $2$ going from a copy of $v_{bottom}$ to the bottom square vertex of the current block which is not equal to the top vertex of the current block. If $e'$ is not a local $F$ edge then we take $e^{*}_b = e'$. We assign half of $e^{*}_b = e$ to $v_{bottom}$ and keep half in reserve in case $e^{*}_b = e^{*}_r$. If $e'$ is a local $F$ edge, let $e''$ be the edge from the top square vertex of the current block to the copy of $v_{bottom}$. If $e''$ is not a local $R$ edge then we take $e^{*}_b = e''$. We assign half of $e^{*}_b = e''$ to $v_{bottom}$ and keep half in reserve in case $e^{*}_b = e^{*}_r$. If $e''$ is a local $R$ edge then we are in the same situation as before so we repeat this process.
\item $e'$ is an edge with label $2$ from the top square vertex in the current block which does not equal the bottom square vertex of the current block to a copy of $v_{bottom}$. In this case, we take $e^{*}_b = e'$. If $e^{*}_b = e'$ is not a local $F$ edge then we assign half of it to $v_{bottom}$ and keep half in reserve in case $e^{*}_b = e^{*}_r$. If $e^{*}_b = e'$ is a local $F$ edge then we assign it to $v_{bottom}$. Note that in this case, we cannot have that $e^{*}_b = e^{*}_r$. The reason for this is that $e^{*}_r$ can only be a local $F$ edge in case $1$, in which case it does not appear again (while $e'$ appears again by definition).
\item $e'$ is an edge with label $2$ which hangs off of a square vertex which is both the top and bottom square vertex for the current block. In this case, we take $e^{*} = e'$ and assign all of it to $v_{bottom}$. Note that in this case, we cannot have that $e^{*}_b = e^{*}_r$.
\item $e'$ is a local $H$ edge with label $1$. In this case, we take $e^{*}_b = e'$ and assign it to $v_{bottom}$. Here we may have that $e^{*}_b = e^{*}_r$ but if we do, one of the endpoints of $e^{*}_b = e'$ is not locally isolated so we can take an edge factor from it and give the edge factor to $v_{right}$.
\end{enumerate}
\end{enumerate}
\end{definition}

\end{document}